\documentclass[psamsfonts,12pt]{amsart}
\usepackage{graphicx}  
\usepackage{amsmath,amssymb,amsthm, amscd}
\usepackage{amssymb,fge}
\usepackage{tikz-cd}

\usepackage{color}

\DeclareMathOperator{\Dom}{Dom}

\DeclareMathOperator{\Orb}{Orb}

\newcommand{\mysetminus}{\mathbin{\fgebackslash}}

\newenvironment{Gal}{{\rm Gal\,}}{}
\newenvironment{Res}{{\rm Res\,}}{}

\newenvironment{Aut}{{\rm Aut}}{}

\newtheorem{theorem}{Theorem}[section]
\newtheorem{lemma}[theorem]{Lemma}
\newtheorem{proposition}[theorem]{Proposition}
\newtheorem{proposition-definition}[theorem]{Proposition-Definition}
\newtheorem{corollary}[theorem]{Corollary}

\newtheorem{question}[theorem]{Question} 

\theoremstyle{definition}
\newtheorem{definition}{Definition}
\theoremstyle{remark}
\newtheorem{remark}{Remark}
\newtheorem{example}{Example}

\usepackage[margin=1in]{geometry}  
\usepackage{graphicx}              
\usepackage{amsmath}               
\usepackage{amsfonts}              
\usepackage{amsthm}                
\usepackage{verbatim}              
\usepackage{indentfirst}           
\usepackage{underscore}
\usepackage{enumitem}
\usepackage{relsize}
\usepackage{varwidth}
\usepackage{mathtools} 
\usepackage{hyperref}
\hypersetup{
    colorlinks=true,
    linkcolor=blue,
    filecolor=magenta,      
    urlcolor=cyan,
}
\usepackage{amssymb}
\usepackage[all]{xy}
\makeatletter
\renewcommand*\env@matrix[1][*\c@MaxMatrixCols c]{%
  \hskip -\arraycolsep
  \let\@ifnextchar\new@ifnextchar
  \array{#1}}
\makeatother

\theoremstyle{remark}

\newcommand*{\Scale}[2][4]{\scalebox{#1}{$#2$}}%

\newcommand{\Mod}[1]{\ (\textup{mod}\ #1)}     

\title[Heights in left, right, and total orbits]{Dynamical height growth: Left, right, and total orbits}   
\author[Wade Hindes]{Wade Hindes}
\begin{document} 
\maketitle
\renewcommand{\thefootnote}{}
\begin{abstract} Let $S$ be a set of dominant rational self-maps on $\mathbb{P}^N$. We study the arithmetic and dynamical degrees of infinite sequences of $S$ obtained by sequentially composing elements of $S$ on the right and left. We then apply this insight to dynamical Galois theory.     
\end{abstract} 
\section{Introduction}
Given a set $S$ of dominant rational maps on $\mathbb{P}^N$ and an infinite sequence $\gamma=(\theta_1,\theta_2, \dots)$ of elements of $S$, then we are interested in two types of iterated processes attached to $\gamma$. Namely, the \emph{left iterative sequence} of maps,   
\[\gamma_n^-:=\theta_n\circ\theta_{n-1}\circ\dots\circ\theta_1\;\;\text{for all $n\geq1$},\] 
and the \emph{right iterative sequence} of maps,
\[\;\;\,\gamma_n^+:=\theta_1\circ\theta_{2}\circ\dots\circ\theta_n\;\;\text{for all $n\geq1$}.\]
In particular, given a suitable initial point $P\in\mathbb{P}^N$ we wish to study the \emph{left and right orbits} of the pair $(\gamma,P)$ given by  
\[\Orb_\gamma^-(P):=\big\{\gamma_n^-(P):n\geq0\big\}\;\;\text{and}\;\; \Orb_\gamma^+(P):=\big\{\gamma_n^+(P):n\geq0\big\}\]
respectively; here we include the identity function $\gamma_0:=\text{Id}_{\mathbb{P}^N}$ for convenience. The analytic and topological properties of these orbits have been previously studied in complex dynamics \cite{random1,random2,random3,random4,random5,random6}, and in this paper, we consider arithmetic analogs of this work. Specifically, if both $P$ and the maps in $S$ are defined over $\overline{\mathbb{Q}}$ and $h:\mathbb{P}^N(\overline{\mathbb{Q}})\rightarrow\mathbb{R}_{\geq0}$ is the absolute Weil height function \cite[\S 3.1]{SilvDyn}, then we are interested in the growth rate of $h(\gamma_n^+(P))$ and $h(\gamma_n^-(P))$ as we move within the left and right orbits of $(\gamma,P)$ respectively. For sets of morphisms, the growth rates of $h(\gamma_n^-(P))$ for left iteration were studied first in \cite{Kawaguchi} and revisited in \cite{stochastic}. In particular, one may construct canonical heights in this setting and recover several familiar facts from the standard theory of arithmetic dynamics \cite{SilvDyn}, where one iterates a single function (i.e., $\gamma$ is a constant sequence). However, there  appears to be relatively little known about heights when iterating on the right. Moreover when $N=1$, the arithmetic properties of $\Orb_\gamma^+(P)$ (for certain $P$ and certain $S$) control the size of the Galois extensions generated by $\gamma_n^+(x)=0$ and $n\geq1$; see Section \ref{sec:Galois}. Therefore, the growth rate of $h(\gamma_n^+(P))$ may be of interest to those studying dynamically generated Galois groups.  
\begin{remark} A further application of our work on left and right orbits is to the growing field of monoid (or semigroup) arithmetic dynamics \cite{monoid1,monoid2,stochastic,IJNT,monoid3,monoid4}. Here, one is instead interested in understanding the arithmetic properties of \emph{total orbits}, 
\begin{equation}\label{eq:totalorbit} 
\Orb_S(P):=\{f(P):\,f\in M_S\}=\bigcup_\gamma \Orb_\gamma^+(P)=\bigcup_\gamma \Orb_\gamma^-(P);
\end{equation} 
here $M_S$ is the monoid generated by $S$ (and the identity) with the operation of composition. However, in practice, if one understands left and right orbits for sufficiently many $\gamma$, then one has gained nontrivial insight into total orbits; for some examples of this heuristic, see \cite[Corollary 1.4]{stochastic}, \cite[Theorem 1.18]{Me:dyndeg}, \cite[Theorem 1.7]{IJNT}, Theorem \ref{thm:zero-one}, and Section \ref{sec:totalorbits}.    
\end{remark} 
As in the case of iterating a single map, some useful tools for analyzing heights in left and right orbits are the left and right dynamical degrees, i.e., the limiting values of $\deg(\gamma_n^-)^{1/n}$ and $\deg(\gamma_n^+)^{1/n}$ respectively. However, without much difficulty, one can construct examples for which the aforementioned limits do not exist \cite[Example 1.1]{Me:dyndeg}. Nevertheless, one expects that these limits converge for most sequences. To test this heuristic, we fix a probability measure $\nu$ on $S$, and extend to a probability measure $\bar{\nu}$ on the set of sequences of elements of $S$ via the product measure; see Section \ref{sec:notation} for more details. With this perspective, we prove that the limits of $\deg(\gamma_n^-)^{1/n}$ and $\deg(\gamma_n^+)^{1/n}$ (as we vary over sequences of $S$) are $\bar{\nu}$-almost surely constant and independent of the direction of iteration. Moreover, for finite sets $S$, we show that this constant bounds both $h(\gamma_n^-(P))^{1/n}$ and $h(\gamma_n^+(P))^{1/n}$ for large $n$; compare to \cite[Theorem 1.8]{Me:dyndeg} and \cite[Theorem 1]{KawaguchiSilverman}. However, to prove this second fact about heights we must enforce a condition on $S$, namely, that as we compose elements of $S$ we manage to avoid maps of degree one: 
\begin{definition} A set of dominant rational maps $S$ on $\mathbb{P}^N$ is called \emph{degree independent} if $\deg(f)\geq2$ for all $f$ in the semigroup generated by $S$; here the operation is composition.   
\end{definition} 
Likewise, since the maps in $S$ may have non-trivial indeterminacy loci, we must take care to ensure that the orbits we consider are actually well defined:
\begin{definition} Let $f$ be in the compositional semigroup generated by $S$, and let $I_f\subset \mathbb{P}^N$ be the indeterminacy locus of $f$. Then we set $\mathbb{P}^N(\overline{\mathbb{Q}})_S:=\displaystyle{\mathbb{P}^N(\overline{\mathbb{Q}})\mysetminus\cup_{f} I_f}$.      
\end{definition} 
With these notions in place, we prove our most general result relating the growth rate of degrees and the growth rate of heights in orbits. The proof is an adaptation and combination of the arguments given for left iteration (only) in Theorems 1.3 and 1.8 of \cite{Me:dyndeg}. Namely, we apply Kingman's subadditive ergodic theorem, Birkhoff's ergodic theorem, and ideas from \cite{SilvermanPN}. In what follows, $\mathbb{E}_\nu[\log\deg(\phi)]=\int_S\log\deg(\phi)d\nu$ denotes the expected value of the random variable $\log\deg$ on $S$.         
\begin{theorem}\label{thm:rationalmaps} Let $S$ be a set of dominant rational self-maps on $\mathbb{P}^N(\overline{\mathbb{Q}})$ and let $\nu$ be a discrete probability measure on $S$. Then the following statements hold:  \vspace{.1cm} 
\begin{enumerate} 
\item[\textup{(1)}] If $\mathbb{E}_\nu[\log\deg(\phi)]$ exists, then there is a constant $\delta_{S,\nu}$ such that the limits \vspace{.1cm} 
\[\lim_{n\rightarrow\infty}\deg(\gamma_n^{-})^{1/n}=\delta_{S,\nu}=\lim_{n\rightarrow\infty}\deg(\gamma_n^{+})^{1/n}\vspace{.1cm}\]
hold (simultaneously) for $\bar{\nu}$-almost every $\gamma\in\Phi_S$. \vspace{.25cm}  
\item[\textup{(2)}] If $S$ is finite and degree independent, then for $\bar{\nu}$-almost every $\gamma\in\Phi_S$ the bounds \vspace{.075cm}
\[\limsup_{n\rightarrow\infty} h(\gamma_n^{\pm}(P))^{1/n}\leq\delta_{S,\nu}\vspace{.075cm}\]
hold (simultaneously) for all $P\in\mathbb{P}^N(\overline{\mathbb{Q}})_S$.  
\end{enumerate}  
\end{theorem} 
Motivated by the existence of the constant $\delta_{S,\nu}$ we make the following definition: 
\begin{definition} For $(S,\nu)$ as in Theorem \ref{thm:rationalmaps}, we call $\delta_{S,\nu}$ the \emph{dynamical degree} of $(S,\nu)$.\end{definition} 
Although Theorem \ref{thm:rationalmaps} gives an upper bound on the growth rate of heights in orbits that is independent of the direction of iteration and the initial point, the same cannot be said in general for lower bounds. Heuristically, if $P$ has small height, then the direction of iteration can matter greatly. We illustrate this  point with the following example.   
\begin{example}{\label{eg:left-right difference}} Let $S=\{x^2-x,3x^2\}$ with $\phi_1=x^2-x$ and $\phi_2=3x^2$, and define $\nu$ on $S$ determined by $\nu(\phi_1)=1/2=\nu(\phi_2)$. Then viewing $S$ as a set of maps on $\mathbb{P}^1$, we consider the possible left and right orbits of $P=1$ and compute that \vspace{.1cm} 
\begin{equation*}
\begin{split} 
\liminf_{n\rightarrow\infty} h(\gamma_n^+(P))^{1/n}=0\;\;&\text{and}\;\;\limsup_{n\rightarrow\infty} h(\gamma_n^+(P))^{1/n}=2\qquad\text{($\bar{\nu}$-almost surely)}\\ 
\;\;\liminf_{n\rightarrow\infty} h(\gamma_n^-(P))^{1/n}=0\;\;&\text{and}\;\;\limsup_{n\rightarrow\infty} h(\gamma_n^-(P))^{1/n}=0\qquad\text{($\bar{\nu}$-probability $1/2$)} \\
\;\;\liminf_{n\rightarrow\infty} h(\gamma_n^-(P))^{1/n}=2\;\;&\text{and}\;\;\limsup_{n\rightarrow\infty} h(\gamma_n^-(P))^{1/n}=2\qquad\text{($\bar{\nu}$-probability $1/2$)} 
\end{split} 
\end{equation*} 
In particular, the direction of iteration may greatly affect the growth rate of heights in orbits.    
\end{example} 
However, for morphisms and sufficiently generic initial points, we are able to prove fairly uniform results. Namely, outside of a set of points $P$ of bounded height, we prove that the limits (not merely the limsups) of both $h(\gamma_n^-(P))^{1/n}$ and $h(\gamma_n^+(P))^{1/n}$ are equal to the dynamical degree, almost surely. Moreover, the dynamical degree is easy to compute for finite sets of morphisms; it is a weighted geometric mean of the degrees of the maps in $S$; compare to \cite[Theorem 1.5]{Me:dyndeg}. The main tools we use to prove this result are Birkhoff's Ergodic Theorem and the Law of Iterated Logarithms for simple random walks; see Section \ref{sec:notation} for statements.   
\begin{theorem}\label{thm:iteratedlogs} Let $S$ be a finite set of endomorphisms of $\mathbb{P}^N(\overline{\mathbb{Q}})$ all of degree at least two, and let $\nu$ be a discrete probability measure on $S$. Then there exists a constant $B_S$ such that the following statements hold: \vspace{.25cm}
\begin{enumerate} 
\item[\textup{(1)}] The dynamical degree is given by $\displaystyle{\delta_{S,\nu}=\prod_{\phi\in S}\deg(\phi)^{\nu(\phi)}}$. \vspace{.3cm} 
\item[\textup{(2)}] For $\bar{\nu}$-almost every $\gamma\in\Phi_S$, the limits \vspace{.1cm}  
\[\lim_{n\rightarrow\infty}h(\gamma_n^-(P))^{1/n}=\delta_{S,\nu}=\lim_{n\rightarrow\infty}h(\gamma_n^+(P))^{1/n}\vspace{.15cm}\]
hold (simultaneously) for all $P$ with $h(P)>B_S$. \vspace{.4cm}  
\item[\textup{(3)}] If the variance $\sigma_{S,\nu}^2$ of $\log(\deg(\phi))$ is nonzero, then for $\bar{\nu}$-almost every $\gamma\in\Phi_S$, \vspace{.1cm}
\[\limsup_{n\rightarrow\infty}\frac{\log\bigg(\mathlarger{\frac{h(\gamma_n^{\pm}(P))}{\delta_{S,\nu}^n}}\bigg)}{\sigma_{S,\nu}\sqrt{2n\log\log n}}=1=\limsup_{n\rightarrow\infty}\frac{\log\bigg(\mathlarger{\frac{\delta_{S,\nu}^n}{h(\gamma_n^{\pm}(P))}}\bigg)}{\sigma_{S,\nu}\sqrt{2n\log\log n}},\vspace{.3cm}\]  
hold (simultaneously) for all $P$ with $h(P)>B_S$. \vspace{.1cm}
\end{enumerate}  
\end{theorem} 
We can rewrite the bounds in Theorem \ref{thm:iteratedlogs} to give improved estimates for $h(\gamma_n^-(P))$ and $h(\gamma_n^+(P))$ that work almost surely. In particular, these bounds have a main term of $\delta_{S,\nu}^n$ and are (at least in an asymptotic sense) independent of both $\gamma$ and $P$; hence, we have reduced the randomness of heights in generic left and right orbits. Specifically, suppose that $S$, $\nu$, $\delta_{S,\nu}$, $B_S$, $\sigma_{S,\nu}^2$ and $P$ satisfy the conditions of the Theorem \ref{thm:iteratedlogs}, and let $\epsilon>0$. Then for almost every $\gamma$ there exists $N_{\gamma,P,\epsilon}$ such that \vspace{.15cm}
\[\delta_{S,\nu}^{\, n-(1+\epsilon)\log_{\delta_{S,\nu}}(e)\,\sigma_{S,\nu}\sqrt{2n\log\log n}}\leq h(\gamma_n^{\pm}(P))\leq \delta_{S,\nu}^{\, n+(1+\epsilon)\log_{\delta_{S,\nu}}(e)\,\sigma_{S,\nu}\sqrt{2n\log\log n}} \vspace{.15cm}\]
holds for all $n\geq N_{\gamma,P,\epsilon}$. It would be interesting to know if and when similar type bounds hold for rational functions; for a conjecture along these lines in the case of iterating a single rational map, see \cite[Conjecture 2]{SilvermanPN}.  

As an application, we can use Theorem \ref{thm:iteratedlogs} to count the number of iterates in left and right orbits of bounded height; compare to \cite[Corollary 1.16]{Me:dyndeg} and \cite[Proposition 3]{KawaguchiSilverman}. \vspace{.1cm}
   
\begin{corollary}\label{cor:escapeptshtbds} Let $S$, $\nu$, $\delta_{S,\nu}$, and $B_S$ be as in Theorem \ref{thm:iteratedlogs}. Then for $\bar{\nu}$-almost every $\gamma\in\Phi_S$ the limits    \vspace{.15cm} 
\[\lim_{B\rightarrow\infty}\frac{\#\{Q\in\Orb_\gamma^-(P)\,:\,h(Q)\leq B\}}{\log(B)}=\frac{1}{\log\delta_{S,\nu}}=\lim_{B\rightarrow\infty}\frac{\#\{W\in\Orb_\gamma^+(P)\,:\,h(W)\leq B\}}{\log(B)} \vspace{.15cm}\] 
hold (simultaneously) for all $P$ with $h(P)>3B_S$.      
\end{corollary}    
Although Theorem \ref{thm:iteratedlogs} and Corollary \ref{cor:escapeptshtbds} give nice descriptions of the growth rate of heights in generic left and right orbits, it is natural to ask what can be said in the non-generic case. Is it possible to prove a result somewhere in-between Theorem \ref{thm:rationalmaps} and Theorem \ref{thm:iteratedlogs}? Likewise, can we prove a result for (suitable) infinite sets $S$? For left iteration of morphisms, we have canonical heights at our disposal \cite{stochastic,Kawaguchi}, but this is not the case when iterating on the right; see Remark \ref{rem:nocanht} below. Moreover, understanding heights in right orbits can be useful for understanding (generalized) dynamical Galois groups; see Section \ref{sec:Galois}. As a first step (with the case of left iteration in mind), we assume that $S$ have further properties, which we now discuss. It is well known that if $\phi:\mathbb{P}^N(\overline{\mathbb{Q}})\rightarrow\mathbb{P}^N(\overline{\mathbb{Q}})$ is a morphism defined over $\overline{\mathbb{Q}}$ of degree $d_\phi$, then 
\begin{equation}\label{functoriality}
h(\phi(P))=d_\phi h(P)+O_{\phi}(1)\;\;\;\text{for all $P\in\mathbb{P}^N(\overline{\mathbb{Q}})$;} \vspace{.1cm} 
\end{equation}  
see, for instance, \cite[Theorem 3.11]{SilvDyn}. With this in mind, we let 
\begin{equation}{\label{htconstant}} 
C(\phi):=\sup_{P \in \mathbb{P}^N(\bar{\mathbb{Q}})} \Big\vert h(\phi(P))-d_\phi h(P)\Big\vert 
\end{equation} 
be the smallest constant needed for the bound in (\ref{functoriality}). Then, in order to control height growth rates for sequences in $S$, we define the following fundamental notion; compare to \cite{stochastic,Me:dyndeg,Kawaguchi}. 
\begin{definition}\label{def:htcontrolled} 
A set $S$ of endomorphisms of $\mathbb{P}^N(\overline{\mathbb{Q}})$ is called \emph{height controlled} if the following properties hold: \vspace{.1cm} 
\begin{enumerate} 
\item $d_S:=\inf\{d_\phi:\phi\in S\}$ is at least $2$. \vspace{.15cm}
\item $C_S:=\sup\{C(\phi): \phi\in S\}$ is finite. \vspace{.1cm}
\end{enumerate} 
\end{definition} 

\begin{remark}We note first that any finite set of morphisms of degree at least $2$ is height controlled. To construct infinite collections, let $T$ be any non-constant set of maps on $\mathbb{P}^1$ and let $S_T=\{\phi\circ x^d\,: \phi\in T,\, d\geq2\}$. Then $S_T$ is height controlled and infinite; a similar construction works for $\mathbb{P}^N$ in any dimension. For another type of example, let $\mathcal{U}$ be the set of roots of unity in $\overline{\mathbb{Q}}$. Then $S=\{x^2+u\,:\, u\in \mathcal{U}\}$ is a height controlled collection of maps on $\mathbb{P}^1$. Moreover, it is worth pointing out that $S$ has a corresponding probability measure given by embedding $\mathcal{U}$ in the unit circle (in $\mathbb{C}$) and then taking the Haar measure on the circle.  
\end{remark} 
With the notion of height control morphisms in place, we prove a result for right iteration in-between Theorem \ref{thm:rationalmaps} and Theorem \ref{thm:iteratedlogs} above; compare to stronger results for left iteration \cite[Theorem 1.2]{stochastic} and \cite[Theorem 1.15]{Me:dyndeg}. However before stating this result, we make a few more notes on the differences between left and right iteration. First, as was mentioned before, canonical heights (in the usual sense) do not exist for right-iteration. That is, in principle one must keep track of both the corresponding liminf and limsup; see statement (1) of Theorem \ref{thm:zero-one} and Remark \ref{rem:nocanht} below. This is a drawback of right-iteration. On the other hand, there are certain advantages as well. For instance, ideally one would like to determine whether or not the total orbit (\ref{eq:totalorbit}) has a certain property by sampling a right or left orbit (and testing that same property). As an example, if a right (or left) orbit of $P$ is finite with positive probability, is it true that $\Orb_S(P)$ is necessarily finite? This statement turns out to be true for right orbits and false for left; for justification, see both Theorem \ref{thm:zero-one} below and \cite[Example 1.10]{Me:dyndeg}.    
  
\begin{theorem}\label{thm:zero-one} Let $S$ be a height controlled set of endomorphisms of $\mathbb{P}^N(\overline{\mathbb{Q}})$ all defined over a fixed number field $K$ and let $\nu$ be a discrete probability measure on $S$. Then the following statements hold: \vspace{.3cm} 
\begin{enumerate}
\item[\textup{(1)}] For all $P$ and all $\gamma$, both 
\[\displaystyle{\liminf_{n\rightarrow\infty}\frac{h(\gamma_n^+(P))}{\deg(\gamma_n^+)}}\;\;\; \text{and}\;\;\;\displaystyle{\limsup_{n\rightarrow\infty}\frac{h(\gamma_n^+(P))}{\deg(\gamma_n^+)}}\] 
exist and are $h(P)+O(1)$. \\[3pt] 
\item[\textup{(2)}] For all $P$, the total orbit $\Orb_S(P)$ of $P$ is infinite if and only if \vspace{.1cm}  
\[\qquad0<\displaystyle{\limsup_{n\rightarrow\infty}\frac{h(\gamma_n^+(P))}{\deg(\gamma_n^+)}}\qquad\qquad\text{($\bar{\nu}$-almost surely).}\vspace{.1cm}\]
Hence, $\Orb_S(P)$ is finite if and only if $\Orb_\gamma^+(P)$ is finite with positive $\bar{\nu}$-probability. \\[3pt]
\item[\textup{(3)}]  If $\Orb_S(P)$ is infinite and $\mathbb{E}_\nu[\log\deg(\phi)]$ exists, then \vspace{.2cm}  
\[\limsup_{n\rightarrow\infty} h(\gamma_n^+(P))^{1/n}=\delta_{S,\nu}\qquad\qquad\text{($\bar{\nu}$-almost surely).}\]
Moreover, the dynamical degree $\delta_{S,\nu}=\exp\big(\mathbb{E}_\nu[\log\deg(\phi)]\big)$ is given explicitly. \vspace{.025cm}  
\end{enumerate}    
\end{theorem} 
\begin{remark}{\label{rem:nocanht}} Note that the $\liminf$ and $\limsup$ in statement (1) of Theorem \ref{thm:zero-one} can be distinct. See Example \ref{eg:left-right difference} above.   
\end{remark} 
Having obtained results for left and right orbits, we turn to height counting problems for total orbits. Intuitively, one expects that if the maps in $S$ are related in some way (for instance, if they commute with each other), then this should cut down the number of possible points in total orbits. More formally, the asymptotic growth rate of the set 
\[\{Q\in\Orb_S(P)\,:\,h(Q)\leq B\}\]
should depend on the structure of the compositional monoid $M_S$ that $S$ generates, at least for generic initial points $P$. As an illustration, we have the following related asymptotic,  
\[\lim_{B\rightarrow\infty}\frac{\#\Big\{f\in M_S\,:\,h\big(f(P)\big)\leq B\Big\}}{(\log B)^s}=\frac{1}{s!\cdot\prod_{i=1}^s\log\deg(\phi_i)}, \vspace{.25cm}\]  
when $S$ is a free basis (of cardinality $s$) for the commutative monoid $M_S$ and $P$ has sufficiently large height. For justification of this fact, as well as a discussion of the problem of counting points of bounded height in total orbits more generally, see Section \ref{sec:totalorbits}. In particular, we discuss how this problem in dynamics relates to the (weighted) growth rate problem for semigroups and to restricted weighted compositions in combinatorics \cite{growth1, compositions, growth2, growth3}.\\[3pt] 

\textbf{Acknowledgements:} We are happy to thank Andrew Bridy, James Douthitt, Joseph Gunther, Vivian Olsiewski Healey, Trevor Hyde, Rafe Jones, and Joseph Silverman for discussions related to the work in this paper. 
\section{Notation and tools from probability}\label{sec:notation} 
We begin by fixing some notation. For more information on these standard constructions in probability, see \cite{Durrett, ProbabilityText}.
\begin{align*} 
S \;\;\;& \text{a set of dominant rational self maps on $\mathbb{P}^N$, all defined over $\overline{\mathbb{Q}}$}.\\ 
\nu \;\;\;& \text{a probability measure on $S$}.\\ 
\Phi_S \;\;& \text{the infinite product $\Phi_S=\Pi_{i=1}^\infty S=S^{\mathbb{N}}$}.\\[2pt]
\bar{\nu} \;\;\;& \text{the product measure $\bar{\nu}=\Pi_{i=1}^\infty \nu$ on $\Phi_S$}. \\
\gamma\;\;\;& \text{an element of $\Phi_S$, viewed as an infinite sequence.}\\ 
\mathbb{E}_{\bar\nu}[f]\;\,& \text{the expected value $\mathlarger{\smallint}_{\hspace{-.1cm}\mathsmaller{\Phi_s}} f\,d\bar{\nu}$ of a random variable $f:\Phi_S\rightarrow\mathbb{R}$}
\end{align*} 
\begin{remark} It is likely that many of our results on dynamical degrees hold without assumptions on the field of definition of the maps in $S$. However, since we wish to study heights, we assume that every map in $S$ has $\overline{\mathbb{Q}}$-coefficients. In particular, the sets $S$ we consider are countable, and for this reason, we assume that $\nu$ is a discrete measure with $\nu(\phi)>0$ for all $\phi\in S$. Likewise, since there may be no natural choice of probability measure $\nu$ on $S$, we keep the measures $\nu$ and $\bar{\nu}$ in much of the notation (e.g., $\mathbb{E}_{\bar\nu}[f]$) to remind the reader of the dependence of our formulas and bounds on the choice of $\nu$. 
\end{remark} 
When $S=\{\phi\}$ is a single map, a crucial tool for establishing the convergence of the limit defining the dynamical degree is Fekete's lemma (see the proof of \cite[Proposition 7]{SilvermanPN}), which states that if $a_n$ is a subadditive sequence of non-negative real numbers, then $\lim a_n/n$ exists. In particular, the following landmark theorem due to Kingman \cite{kingman} may be viewed as a random version of Fekete's lemma. In what follows, the expected value $\mathbb{E}_{\mu}[f]$ of a random variable $f: \Omega\rightarrow \mathbb{R}$ on a probability space $(\Omega,\Sigma, \mu)$ is the integral $\int_\Omega f d\mu$.      
\begin{theorem}[Kingman's Subadditive Ergodic Theorem]\label{thm:kingman} Let $T$ be a measure preserving transformation on a probability space $(\Omega,\Sigma, \mu)$, and let $(g_n)_{n\geq1}$ be a sequence of $L^1$ random variables that satisfy the subadditivity relation  
\begin{equation}\label{subadd} 
g_{m+n}\leq g_n+g_m\circ T^n   
\end{equation} 
for all $n,m\geq1$. Then there exists a $T$-invariant function $g$ such that  
\[\lim_{n\rightarrow\infty}\frac{g_n(x)}{n}=g(x)\]  
for $\mu$-almost every $x\in\Omega$. Moreover, if $T$ is ergodic, then $g$ is constant and \vspace{.1cm} 
\[\lim_{n\rightarrow\infty}\frac{g_n(x)}{n}=\lim_{n\rightarrow\infty}\frac{\mathbb{E}_\mu[g_n]}{n} =\inf_{n\geq1}\frac{\mathbb{E}_\mu[g_n]}{n}.\]    
for $\mu$-almost every $x\in\Omega$
\end{theorem} 
\begin{remark} A transformation $T:\Omega\rightarrow\Omega$ on a probability space $(\Omega,\Sigma,\mu)$ is called \emph{ergodic} if for all $E\in\Sigma$ such that $T^{-1}(E)=E$, either $\mu(E)=0$ or $\mu(E)=1$.   
\end{remark} 
We also need a similar (yet weaker) ergodic theorem due to Birkhoff. 
\begin{theorem}[Birkhoff's Ergodic Theorem]\label{birk} If $T$ is an ergodic, measure preserving transformation on a probability space $(\Omega,\Sigma, \mu)$, then for every random variable $f\in L^1(\Omega)$,
\begin{equation}\label{birkhoff} \lim_{n\rightarrow\infty} \frac{1}{n}\sum_{j=0}^{n-1} f\circ T^j(x)=\mathbb{E}_\mu[f].  
\end{equation}    
for $\mu$-almost every $x\in\Omega$. 
\end{theorem}

To apply Kingman's Subadditive Ergodic Theorem to dynamical degrees, we use the following well known example of an ergodic, measure preserving transformation. In particular, the lemma below is a simple consequence of Kolmogorov's $0$\,-$1$ law \cite[Theorem 10.6]{ProbabilityText}; for nice further discussions, see \cite[Example 7.1.6]{Durrett} or \cite[Example 5.5]{steve2} and \cite[Exercise 5.11]{steve2}.    
\begin{lemma}\label{shift} Let $S$ be a set with probability measure $\nu$ and let $(\Phi_S,\bar{\nu})$ be the corresponding infinite product space. Then the shift map, 
\[T\big(\theta_1,\theta_2, \dots \big)=(\theta_2, \theta_3,\dots)\] 
is an ergodic, measure preserving transformation on $\Phi_S$.   
\end{lemma}
\begin{remark} When $S$ is a finite set, the probability space $\Phi_S$ and the map $T$ as in Lemma \ref{shift} are often called Bernoulli schemes and Bernoulli shifts respectively.  
\end{remark} 
Finally, to obtain the improved height bounds in part (3) of Theorem \ref{thm:iteratedlogs} with a main term of $\delta^n$, we use the following result due to Hartman and Wintner known as the Law of Iterated Logarithms; see  \cite[Theorem 8.11.3]{Durrett}. As with certain classical theorems in probability (e.g., the Law of Large Numbers, The Central Limit Theorem, etc.) the Law of Iterated Logarithms for simple random walks is normally stated in terms of independent and identically distributed (or \emph{i.i.d.} for short) random variables; see \cite[\S2.1]{Durrett} or \cite[\S10]{ProbabilityText} for a definition and discussing of i.i.d sequences. However, for our purposes, it suffices to know that if $f:S\rightarrow\mathbb{R}$ is any $\nu$-measurable function, then the corresponding projection maps $X_n(f):\Phi_S\rightarrow\mathbb{R}$ on the product space $(\Phi_S,\bar{\nu})$ given by $X_{n,f}(\theta_1,\theta_2, \dots)=f(\theta_n)$ form an i.i.d sequence of random variables; this is a simple consequence of the relevant definitions \cite[Corollary 10.2]{ProbabilityText}.     
\begin{theorem}[Law of Iterated Logarithms]\label{thm:lawiterlogs} Suppose that $X_1$, $X_2$, $\dots$ are i.i.d. random variables on $(\Omega,\Sigma, \mu)$ with $\mathbb{E}_\mu[X_i]=0$ and $\mathbb{E}_\mu[X_i^2]=1$. Then, if $S_n=X_1+\dots+X_n$ denotes the truncated sum, we have that   
\begin{equation}\label{brownian}
\qquad\limsup_{n\rightarrow\infty} \frac{\pm S_n}{\sqrt{2n\log\log n}}=1\qquad\text{($\mu$-almost surely).}
\end{equation}  
\end{theorem} 
\begin{remark} Interestingly, the Law of Iterated Logarithms (for simple random walks) stated above is proven by first establishing the analogous fact for Brownian motion and then deducing (\ref{brownian}) from that case.  
\end{remark} 
\section{Rational maps: dynamical degrees and height bounds}
In this section, we prove Theorem \ref{thm:rationalmaps} on dynamical degrees and height bounds for rational maps; for strengthened results on morphisms, see Section \ref{sec:morphisms}.  
\begin{proof}[(Proof of Theorem \ref{thm:rationalmaps})] We begin with the proof of statement (1) on dynamical degrees. For $n\geq1$, we define the random variables $g_n^{-}:\Phi_S\rightarrow\mathbb{R}_{\geq0}$ and $g_n^{+}:\Phi_S\rightarrow\mathbb{R}_{\geq0}$ given by 
\[g_n^{-}(\gamma):=\log\deg(\gamma_n^{-})\;\;\text{and}\;\;g_n^{+}(\gamma):=\log\deg(\gamma_n^{+})\] 
respectively. Note that each $g_n^{\pm}$ is non-negative since $S$ is a collection of dominant maps. We will show that the sequences $(g_n^-)_{n\geq1}$ and $(g_n^+)_{n\geq1}$ satisfy the hypothesis of Kingman's Subadditive Ergodic Theorem. Note first that each $g_n^{\pm}$ factors through the finite product $S^n$ and $S^n$ (a countable set) is equipped with the discrete measure (a finite product of discrete spaces is discrete). In particular, $g_n^{\pm}$ is $\bar{\nu}$-measurable by \cite[Corollary 10.2]{ProbabilityText}. Likewise, define $f_i:\Phi_S\rightarrow\mathbb{R}_{\geq0}$ given by $f_i(\gamma)=\log\deg(\theta_i)$ for $\gamma=(\theta_s)_{s=1}^\infty$. Then $f_i$ is also measurable by \cite[Corollary 10.2]{ProbabilityText}. Moreover, we see that $g_n^{\pm}\leq\sum_{i=1}^nf_i$, since 
\begin{equation}\label{degbd} 
\deg(F\circ G)\leq\deg(F)\deg(G)\;\;\;\;\text{for any}\; F,G\in \Dom(\mathbb{P}^N);
\end{equation} 
here, $\Dom(\mathbb{P}^N)$ is the set of dominant self-maps on $\mathbb{P}^N$. In particular,   
\[\mathbb{E}_{\bar{\nu}}[g_n^{\pm}]\leq\sum_{i=1}^n\mathbb{E}_{\bar{\nu}}[f_i]=n\,\mathbb{E}[f_1]:=n\,\mathbb{E}_\nu[\log\deg(\phi)];\] 
here we use that $\Phi_S$ consists of i.i.d sequences. In particular, each $g_n^{\pm}$ is an $L^1$ function since $\mathbb{E}_\nu[\log\deg(\phi)]$ is bounded by assumption. Now we check the subadditivity relation in (\ref{subadd}), a simple consequence of (\ref{degbd}). Let $n,m>0$, let $\gamma=(\theta_s)_{s=1}^\infty$, and let $T$ be the shift map on $\Phi_S$. Then we compute that \vspace{.25cm}  
\begin{equation*}
\begin{split} 
g_{n+m}^{\,-}(\gamma)=\log\deg(\theta_{m+n}\circ\dots\circ\theta_1)&\leq\log\deg(\theta_{m+n}\circ\dots\circ\theta_{n+1})+\log\deg(\theta_n\circ\dots\circ\theta_1)\\[3pt] 
&=g_m^-(T^n(\gamma))+g_n^-(\gamma)=g_n^-(\gamma)+g_m^-(T^n(\gamma)), \vspace{.25cm}
\end{split} 
\end{equation*}    
by (\ref{degbd}). Likewise for right iteration, we see that \vspace{.1cm}  
\begin{equation*}
\begin{split} 
g_{n+m}^{\,+}(\gamma)=\log\deg(\theta_{1}\circ\dots\circ\theta_{n+m})&\leq\log\deg(\theta_{1}\circ\dots\circ\theta_{n})+\log\deg(\theta_{n+1}\circ\dots\circ\theta_{n+m})\\[3pt] 
&=g_n^+(\gamma)+g_m^+(T^n(\gamma)), \vspace{.25cm}
\end{split} 
\end{equation*} 
In particular, Theorem \ref{thm:kingman} and Lemma \ref{shift} together imply that \vspace{.2cm} 
\begin{equation}\label{kinglim} 
\lim_{n\rightarrow\infty}\log\deg(\gamma_n^{\pm})^{1/n}=\lim_{n\rightarrow\infty}\frac{g_n^{\pm}(\gamma)}{n}=\lim_{n\rightarrow\infty}\frac{\mathbb{E}_{\bar{\nu}}[g_n^{\pm}]}{n}=\inf_{n\geq1}\frac{\mathbb{E}_{\bar{\nu}}[g_n^{\pm}]}{n}
\end{equation} 
for $\bar{\nu}$-almost every $\gamma\in\Phi_S$. However, apriori the limits 
\[\delta_{S,\nu}^-:=\lim_{n\rightarrow\infty}\frac{\mathbb{E}_{\bar{\nu}}[g_n^{-}]}{n}\;\;\text{and}\;\; \delta_{S,\nu}^+:=\lim_{n\rightarrow\infty}\frac{\mathbb{E}_{\bar{\nu}}[g_n^{+}]}{n}\] 
could be distinct (in fact, if we were to allow maps over $\mathbb{C}$ so that $S$ could be uncountable, then we expect that this could be the case). But $S$ is countable and discrete by assumption, and so these limits are in fact equal. To see this, we define the bijections $\tau_n:S^n\rightarrow S^n$ given by 
\[\tau_n(\theta_1,\dots,\theta_n)=(\theta_n,\dots,\theta_1)\] 
and let $\nu_n=\nu\times\dots\times\nu$ be the product probability measure on $S^n$. Then it follows from the definition of $\nu_n$ that 
\[\nu_n(\theta_1,\dots,\theta_n)=\nu(\theta_1)\cdots\nu(\theta_n)=\nu(\theta_n)\cdots\nu(\theta_1)=\nu_n(\tau_n(\theta_1,\dots\theta_n))\]
see \cite[\S10]{ProbabilityText}. Now let $G_n^{\pm}$ be the random variables on $S^n$ given by \vspace{.1cm} 
\[G_n^-(\theta_1,\dots,\theta_n)=\log\deg(\theta_n\circ\dots\circ\theta_1)\;\;\;\text{and}\;\;\;G_n^+(\theta_1,\dots,\theta_n)=\log\deg(\theta_1\circ\dots\circ\theta_n)\vspace{.1cm}\] 
In particular, it is straightforward to check that $G_n^-=G_n^+\circ\tau_n$. Therefore, since $S^n$ is countable/discrete, $\tau_n$ is bijective, and the series below are absolutely convergent:\vspace{.1cm}   
\begin{equation}{\label{eq:directionswap}}
\mathbb{E}_{\nu_n}[G_n^{-}]=\sum_{x\in S^n}G_n^-(x)\nu_{n}(x)=\sum_{x\in S^n}G_n^+(\tau_n(x))\nu_{n}(\tau(x))=\sum_{y\in S^n}G_n^+(y)\nu_{n}(y)=\mathbb{E}_{\nu_n}[G_n^{+}].\vspace{.1cm} 
\end{equation}
On the other hand, $g_n^{\pm}$ factors through $G_n^{\pm}$, so that \cite[Theorem 10.4]{ProbabilityText} and (\ref{eq:directionswap}) together imply that \vspace{.1cm}
\begin{equation}\label{eq:swap2} 
\mathbb{E}_{\bar{\nu}}[g_n^{-}]=\mathbb{E}_{\nu_n}[G_n^{-}]=\mathbb{E}_{\nu_n}[G_n^{+}]=\mathbb{E}_{\bar{\nu}}[g_n^{+}]\qquad\text{for all $n\geq1$}. \vspace{.1cm}
\end{equation}  
Hence, it follows from (\ref{kinglim}) and (\ref{eq:swap2}) that  
\begin{equation}{\label{eq:swap3}}
\lim_{n\rightarrow\infty}\log\deg(\gamma_n^{-})^{1/n}=\lim_{n\rightarrow\infty}\frac{\mathbb{E}_{\bar{\nu}}[g_n^{-}]}{n}=\lim_{n\rightarrow\infty}\frac{\mathbb{E}_{\bar{\nu}}[g_n^{+}]}{n}=\lim_{n\rightarrow\infty}\log\deg(\gamma_n^{+})^{1/n}
\end{equation} 
for $\bar{\nu}$-almost every $\gamma\in\Phi_S$; here we use also that the intersection of almost sure events is almost sure. Moreover, applying the exponential map to (\ref{eq:swap3}) and exchanging $\exp$ with the limit (justified, by continuity) gives 
\begin{equation}\label{eq:dendegdef}
\lim_{n\rightarrow\infty}\deg(\gamma_n^{\pm})^{1/n}=\delta_{S,\nu}:=\exp\Big(\lim_{n\rightarrow\infty}\frac{\mathbb{E}_{\bar{\nu}}[g_n^{-}]}{n}\Big)=\exp\Big(\lim_{n\rightarrow\infty}\frac{\mathbb{E}_{\bar{\nu}}[g_n^{+}]}{n}\Big)
\end{equation} 
for $\bar{\nu}$-almost every $\gamma\in\Phi_S$ as claimed. 

Now for the proof of statement (2) of Theorem \ref{thm:rationalmaps}. Suppose that $S$ is finite and degree independent. Let $k\geq1$ be an integer, and let 
\begin{equation}\label{def:strings} 
M_{S,k}:=\big\{f=\theta_1\circ\dots\circ\theta_k\,\big\vert\;\text{for some}\;(\theta_1,\dots,\theta_k)\in S^k\big\}
\end{equation} 
be the set of possible functions generated by $k$-term strings of elements of $S$. Then a standard triangle inequality estimate (see the proof of \cite[Theorem 3.11]{SilvDyn}) implies that 
\begin{equation}\label{rat:bd1} 
\,h(f(Q))\leq\deg(f) \,h(Q)+C(k,S)\qquad \text{for all $f\in M_{S,k}$ and all $Q\in\mathbb{P}^N(\overline{\mathbb{Q}})_S$}.  
\end{equation}  
To see this, note that there is such a constant for each $f$ and only finitely many $f$'s, since $S$ is a finite set. Moreover, it is important to note that the estimate above does not depend on the direction of iteration (but simply the length of the string). In particular, we see that if $P\in \mathbb{P}^N(\overline{\mathbb{Q}})_S$, if $n\geq1$, and if  $F_{nk}=f_n\circ f_{n-1}\circ\dots\circ f_1$ is an arbitrary element of $M_{S,nk}$ for some choice of $f_i\in M_{S,k}$, then repeated application of the bound in (\ref{rat:bd1}) implies that \vspace{.25cm}  
\begin{equation}\label{eq:stringbd} 
\begin{split}
h(F_{nk}(P))\leq&\deg(f_n)\deg(f_{n-1})\dots\deg(f_1)\Scale[.84]{\Big(h(P)+\frac{C(k,S)}{\deg(f_1)}+\frac{C(k,S)}{\deg(f_1)\deg(f_2)}+\dots+\frac{C(k,S)}{\deg(f_1)\dots\deg(f_n)}\Big)} \\[5pt]
\leq&\deg(f_n)\deg(f_{n-1})\dots\deg(f_1) \Big(h(P)+C(k,S)\Big). \vspace{.15cm}  
\end{split} 
\end{equation}
Here we use our assumption that $S$ is degree independent, so that $\deg(f_i)\geq2$ for all $i$. Now we apply this bound to sequences. For $\gamma=(\theta_s)_{s=1}^{\infty}\in\Phi_S$ and $i,k\geq1$, let \vspace{.15cm}   
\[f_{i,k}^-(\gamma)=\theta_{ik}\circ\theta_{ik-1}\circ\dots\circ\theta_{(i-1)k+1}\;\;\;\text{and}\;\;\;f_{i,k}^+(\gamma)=\theta_{(i-1)k+1}\circ \theta_{(i-1)k+1}\dots\theta_{ik}. \vspace{.15cm}  \]    
In particular, it is straightforward to check that \vspace{.15cm}  
\[\gamma_{nk}^-=f_{n,k}^-(\gamma)\circ f_{n-1,k}^-(\gamma)\circ\dots \circ f_{1,k}^-(\gamma)\;\;\; \text{and}\;\;\;\gamma_{nk}^+=f_{1,k}^+(\gamma)\circ f_{2,k}^+(\gamma)\circ\dots\circ f_{n,k}^+(\gamma). \vspace{.15cm}  \] 
Moreover, each $f_{i,k}^{\pm}(\gamma)\in M_{S,k}$ is the composition of a $k$-term string from $S$. Therefore, (\ref{eq:stringbd}) above applied separately to $F_{nk}=\gamma_{nk}^-$ and $F_{nk}=\gamma_{nk}^+$ implies that \vspace{.15cm}  
\begin{equation}\label{rat:bd2} 
\begin{split} 
 h(\gamma_{nk}^-(P))\leq\deg(f_{1,k}^-(\gamma))\deg(f_{2,k}^-(\gamma))\dots\deg(f_{n,k}^-(\gamma)) \,C(k,S,P)\\[8pt]
 h(\gamma_{nk}^+(P))\leq\deg(f_{1,k}^+(\gamma))\deg(f_{2,k}^+(\gamma))\dots\deg(f_{n,k}^+(\gamma))\,C(k,S,P)
\end{split}  
\end{equation}
holds for all $n,k\geq1$, all $\gamma\in\Phi_{S}$ and all $P\in\mathbb{P}^N(\overline{\mathbb{Q}})_S$; here we reverse the order of the product of the degrees for left iteration,\vspace{.1cm}   
\[\deg(f_{n,k}^-(\gamma))\deg(f_{n-1,k}^-(\gamma))\dots\deg(f_{1,k}^-(\gamma))=\deg(f_{1,k}^-(\gamma))\deg(f_{2,k}^-(\gamma))\dots\deg(f_{n,k}^-(\gamma)),\vspace{.1cm}  \]
to streamline the argument to come. From here we use Birkhoff's Ergodic Theorem to control the right hand side of (\ref{rat:bd2}) above. Namely, let $T_{(k)}:\Phi_S\rightarrow\Phi_S$ denote the $k$-shift map, $T_{(k)}:=T^k=T\circ T\circ\dots\circ T$. In particular, since the shift map $T$ is ergodic and measure preserving by Lemma \ref{shift}, so is $T_{(k)}$ for all $k\geq1$. Now consider the random variables $F_{(k)}^{-}:\Phi_S\rightarrow\mathbb{R}_{\geq0}$ and $F_{(k)}^{+}:\Phi_S\rightarrow\mathbb{R}_{\geq0}$ given by \vspace{.1cm} 
\[ F_{(k)}^{\pm}(\gamma)=\frac{\log\deg(\gamma_k^{\pm})}{k}=\frac{\log\deg(f_{1,k}^{\pm}(\gamma))}{k}\;\;\;\;\ \text{for $\gamma\in\Phi_S$}.\]
Then, it follows from the definition of $f_{i,k}^\pm$ that $F_{(k)}^{\pm}\circ T_{(k)}^{i-1}=1/k\cdot\log\deg(f_{i,k}^{\pm})$. Hence, rewriting the bounds in (\ref{rat:bd2}) and taking $nk$-th roots, we see that \vspace{.1cm} 
\begin{equation}\label{rat:bd3}
h(\gamma_{nk}^{\pm}(P))^{1/nk}\leq \bigg(\exp\frac{1}{n}\sum_{j=0}^{n-1}F_{(k)}^{\pm}\big(T_{(k)}^j(\gamma)\big)\bigg)
\,C(k,S,P)^{1/nk}.\end{equation}   
In particular, (\ref{rat:bd3}) implies that 
\begin{equation}\label{rat:bd4}
\limsup_{n\rightarrow\infty}h(\gamma_{nk}^{\pm}(P))^{1/nk}\leq\limsup_{n\rightarrow\infty}\bigg(\exp\,\frac{1}{n}\sum_{i=0}^{n-1}F_{(k)}^{\pm}\big(T_{(k)}^i(\gamma)\big)\bigg).
\end{equation} 
However, Birkhoff's Ergodic Theorem \ref{birk} implies that 
\begin{equation}\label{rat:lim}
\lim_{n\rightarrow\infty}\frac{1}{n}\sum_{i=0}^{n-1}F_{(k)}^{\pm}\big(T_{(k)}^i(\gamma)\big)=\mathbb{E}_{\bar{\nu}}[F_{(k)}^{\pm}]
\end{equation} 
for almost every $\gamma\in\Phi_{S}$; note that this claim is independent of the point $P$.  Moreover, since a countable intersection of almost sure events is almost sure, we see that the limit in (\ref{rat:lim}) is \textbf{true for all k} (for both left and right iteration), for almost every $\gamma\in\Phi_S$. On the other hand, (\ref{eq:swap2}) above implies that  
\begin{equation}\label{eq:exp=}
\mathbb{E}_{\bar{\nu}}[F_{(k)}^{-}]=\frac{\mathbb{E}_{\bar{\nu}}[g_k^{-}]}{k}=\frac{\mathbb{E}_{\bar{\nu}}[g_k^{+}]}{k}=\mathbb{E}_{\bar{\nu}}[F_{(k)}^{+}].
\end{equation} 
Hence, the limit on the righthand side of (\ref{rat:lim}) does not depend on the direction. Therefore, (\ref{rat:bd4}), (\ref{rat:lim}), and the fact that the exponential function is continuous together imply that \vspace{.1cm} 
\begin{equation}\label{rat:bigbd}
\limsup_{n\rightarrow\infty}h(\gamma_{nk}^{\pm}(P))^{1/nk}\leq \exp\bigg(\mathbb{E}_{\bar{\nu}}\Big[\frac{\log\deg(\gamma_k^-)}{k}\Big]\bigg) \vspace{.1cm}  
\end{equation}
holds for all $k$ (for both left and right iteration), for almost every $\gamma\in\Phi_S$.

From here, we handle left and right iteration separately and begin with left iteration. In particular, we show that the overall limsup (without $k$) in part (2) of Theorem \ref{thm:rationalmaps} can be computed using the subsequence of multiples of $k$ (for any $k\geq1$). This line of reasoning does not work for right iteration in general; see Example \ref{eg:left-right difference}. To do this, define constants
\begin{equation}\label{rat:degbd} 
d_{S,k}:=\max_{\substack{f\in M_{S,r}\\ 0\leq r<k}}\deg(f)\;\;\; \text{and}\;\;\;\ B_{S,k}:=\max_{0\leq r<k} C(r,S);
\end{equation} 
here, we remind the reader that $C(r,S)$ is the height bound constant given by \vspace{.1cm}
\begin{equation}\label{rat:degbd2}
C(r,S)=\max_{f\in M_{S,r}}\sup_{Q\in\mathbb{P}^N(\overline{\mathbb{Q}})}\{h(f(Q))-\deg(f)h(Q)\}. \vspace{.1cm}
\end{equation} 
In particular, both $d_{S,k}$ and $B_{S,k}$ are finite since $S$ is a finite set. From here we proceed as in the proof of \cite[Proposition 12]{SilvermanPN}. Namely, for any $k\geq1$ and $m\geq k$, we can write $\gamma_m^-=f\circ\gamma_{nk}^-$ for some $f\in M_{S,r}$, some $0\leq r< k$, and some $n\geq1$. With this in mind,  \vspace{.15cm}   
\begin{equation}\label{rat:subseq}
\begin{split}
\limsup_{m\rightarrow\infty} h(\gamma_m^{-}(P))^{1/m}&=\limsup_{n\rightarrow\infty} \max_{0\leq r<k} h(\gamma_{r+nk}^{-}(P))^{1/(r+nk)}\\[5pt]
&\leq\limsup_{n\rightarrow\infty} \Big(d_{S,k}\,h(\gamma_{nk}^{-}(P))+B_{S,k}\Big)^{1/nk}\;\;\;\;\;\ \text{by (\ref{rat:bd1}),  (\ref{rat:degbd}), and (\ref{rat:degbd2})}\\[5pt] 
&=\limsup_{n\rightarrow\infty} h(\gamma_{nk}^{-}(P))^{1/nk} \vspace{.1cm}
\end{split} 
\end{equation}
Hence, combining the bound in (\ref{rat:bigbd}) with (\ref{rat:subseq}), we see that \vspace{.2cm}
\begin{equation}{\label{rat:bd6}}
\limsup_{m\rightarrow\infty} h(\gamma_m^{-}(P))^{1/m}\leq \exp\bigg(\mathbb{E}_{\bar{\nu}}\Big[\frac{\log\deg(\gamma_k^-)}{k}\Big]\bigg)=\exp\bigg(\frac{\mathbb{E}_{\bar{\nu}}[\log\deg(\gamma_k^-)]}{k}\bigg) \vspace{.2cm}
\end{equation}
holds for all $k\geq1$, for all $P\in\mathbb{P}^N(\bar{\mathbb{Q}})_S$, for $\bar{\nu}$-almost every $\gamma\in\Phi_S$. 
 
Now for iteration on the right. For any $k\geq1$ and $m\geq k$, write $\gamma_m^+=\gamma_{nk}^+\circ f$ for some $f\in M_{S,r}$, some $0\leq r<k$, and some $n\geq1$. Now let 
\[M_{S,k}(P):=\big\{Q\in\mathbb{P}^N(\overline{\mathbb{Q}})\,:\, Q=f(P)\;\text{for some $f\in M_{S,r}$ and $0\leq r<k$} \big\}\] 
In particular, $M_{S,k}(P)$ is a finite set of points since $S$ is finite. Therefore, 
\[\mathcal{C}_{S,k,P}:=\max_{Q\in M_{S,k}(P)}\{h(Q)+B_{S,k}\}\]
is a finite constant. Moreover, $h(\gamma_m^+(P))=h(\gamma_{nk}^+(Q))$ for some $Q\in M_{S,k}(P)$ by construction. On the other hand, (\ref{eq:stringbd}) and (\ref{rat:bd3}) hold for all $P\in\mathbb{P}^N(\overline{\mathbb{Q}})_S$. In particular, these bounds hold for all $Q\in M_{S,k}(P)$. Therefore, 
\begin{equation}{\label{rat:bd7}}
h(\gamma_m^+(P))^{1/m}=h(\gamma_{nk}^+(Q))^{1/m}\leq h(\gamma_{nk}^+(Q))^{1/nk}\leq \bigg(\exp\frac{1}{n}\sum_{j=0}^{n-1}F_{(k)}^{+}\big(T_{(k)}^j(\gamma)\big)\bigg)
\,\mathcal{C}_{S,k,P}^{1/nk}
\end{equation}
As before letting $m\rightarrow\infty$ (and therefore $n\rightarrow\infty$), Birkhoff's Ergodic Theorem implies that  
\begin{equation}{\label{rat:bd8}} 
\limsup_{m\rightarrow\infty} h(\gamma_m^{+}(P))^{1/m}\leq \exp\bigg(\mathbb{E}_{\bar{\nu}}\Big[\frac{\log\deg(\gamma_k^-)}{k}\Big]\bigg)=\exp\bigg(\frac{\mathbb{E}_{\bar{\nu}}[\log\deg(\gamma_k^-)]}{k}\bigg) \vspace{.2cm}
\end{equation}
holds for all $k\geq1$, for all $P\in\mathbb{P}^N(\bar{\mathbb{Q}})_S$, for $\bar{\nu}$-almost every $\gamma\in\Phi_S$; recall that the limit of the expected values of $F_{(k)}^-$ and $F_{(k)}^+$ are equal by (\ref{eq:exp=}). In particular letting $k\rightarrow\infty$, we deduce from (\ref{eq:dendegdef}) and our combined bounds in (\ref{rat:bd6}) and (\ref{rat:bd8}), that for $\bar{\nu}$-almost every $\gamma\in\Phi_S$ the bounds \vspace{.1cm}
\[\limsup_{m\rightarrow\infty} h(\gamma_m^{\pm}(P))^{1/m}\leq \delta_{S,\nu}\]
hold (simultaneously) for all $P\in\mathbb{P}^N(\overline{\mathbb{Q}})_S$. This completes the proof of Theorem \ref{thm:rationalmaps}.  
\end{proof}

\section{Morphisms: dynamical degrees and height bounds}\label{sec:morphisms}
Throughout this section, let $S$ be a set of height controlled set of endomorphisms on $\mathbb{P}^N$. Ideally, one would like to strengthen part (2) of Theorem \ref{thm:rationalmaps} for rational maps in two ways: to replace the limsup with a limit, and to replace the inequality with an equality; compare to \cite[Conjecture 6.d]{KawaguchiSilverman} and \cite[Conjecture 1.b]{SilvermanPN}. We succeed in proving this when $S$ is a finite set and the initial point $P$ has sufficiently large height. Moreover (perhaps surprisingly), the resulting limit is (almost surely) independent of the direction of iteration. To prove both Theorems \ref{thm:iteratedlogs} and \ref{thm:zero-one}, we need the following generalization of Tate's telescoping argument. In what follows, $M_S$ is the monoid generated by $S$ under composition, and $d_S$ and $C_S$ are the height controlled constants in Definition \ref{def:htcontrolled}. 
\begin{lemma}{\label{lem:tate}} Let $S$ be a height controlled set of endomorphisms of $\mathbb{P}^N(\overline{\mathbb{Q}})$, and let $d_S$ and $C_S$ be the corresponding height controlling constants. Then for all $\rho\in M_S$, 
\[\bigg|\frac{h(\rho(Q))}{\deg(\rho)}-h(Q)\bigg|\leq \frac{C_S}{d_S-1} \;\;\;\; \text{for all $Q\in \mathbb{P}^N(\overline{\mathbb{Q}})$.}\] 
\end{lemma} 
\begin{proof} Suppose that $\rho=\theta_r\circ\theta_{r-1}\dots \circ\theta_1$ for $\theta_i\in S$, and let $\theta_0$ to be the identity map on $\mathbb{P}^N$. Then define 
\[\rho_{i}:=\theta_i\circ\theta_{i-1}\dots \circ\theta_1\circ\theta_0 \;\;\;\; \text{for $0\leq i\leq r$}.\vspace{.05cm}\]
Note, that $\rho=\rho_r$ and $\rho_0=\theta_0$ is the identity map. In particular, inspired by Tate's telescoping argument, we rewrite \vspace{.05cm} 
\begin{equation}{\label{Tate}}
\begin{split}
\bigg|\frac{h(\rho(Q))}{\deg(\rho)}-h(Q)\bigg|&=\bigg|\sum_{i=0}^{r-1}\frac{h(\rho_{r-i}(Q))}{\deg(\rho_{r-i})}- \frac{h(\rho_{r-i-1}(Q))}{\deg(\rho_{r-i-1})}\bigg|\\[5pt]
&\leq \sum_{i=0}^{r-1}\bigg|\frac{h(\rho_{r-i}(Q))}{\deg(\rho_{r-i})}- \frac{h(\rho_{r-i-1}(Q))}{\deg(\rho_{r-i-1})}\bigg| \\[5pt]
&=\sum_{i=0}^{r-1}\frac{\Big|h(\rho_{r-i}(Q))-\deg(\theta_{r-i})h(\rho_{r-i-1}(Q))\Big|}{\deg(\rho_{r-i})} \\[5pt] 
&\leq \sum_{i=1}^{r}\frac{C}{(d_S)^{i}}\leq \sum_{i=1}^{\infty}\frac{C_S}{(d_S)^i}=\frac{C_S}{d_S-1}.  
\end{split} 
\end{equation}  
This completes the proof of Lemma \ref{lem:tate}.
\end{proof} 
With this height bound in place, we are nearly ready to prove our main result for sets of  morphisms, Theorem \ref{thm:iteratedlogs}. In fact, we are able to prove a stronger result. Namely, both $h(\gamma_n^{\pm}(P))^{1/n}$ approach the dynamical degree (almost surely) whenever $P$ is a so called escape point for $S$; see Definition \ref{def:escapepts} below. Moreover, every point $P$ of sufficiently large height is an escape point for $S$, and we therefore recover Theorem \ref{thm:iteratedlogs}. 
\begin{remark} This improved version can be useful for analyzing dynamical Galois groups; see Section \ref{sec:Galois}. For instance, if $S=\{x^{d_1}+c_1, \dots,  x^{d_s}+c_s\}$ is a set of unicritical polynomials, then the right orbits of $P=0$ (i.e., the critical orbits) control the ramification in the associated towers of splitting fields; see Proposition \ref{prop:discriminant} below. However, $P=0$ does not have large enough height to apply Theorem \ref{thm:iteratedlogs} directly. Nevertheless, $P=0$ is very often an escape point for $S$ (see Corollary \ref{cor:unicritescape} below), in which case the conclusions of Theorem \ref{thm:iteratedlogs} still hold.      
\end{remark} 
To define escape points, recall that $M_{S,r}$ denotes the set of functions generated by tuples of elements of $S$ of length $r$; see (\ref{def:strings}) above. Moreover by convention, $M_{S,0}$ is the singleton set containing the identity function.  
\begin{definition}\label{def:escapepts} Let $S$ be a height controlled set of endomorphisms of $\mathbb{P}^N(\overline{\mathbb{Q}})$ and define  $B_S:=C_S/(d_S-1)$. If there exists $r\geq0$ such that $h(g(P))> B_S$ for all $g\in M_{S,r}$,
then we say that $P$ is an \emph{escape point} for $S$. Moreover, we call the minimum such value of $r$ the \emph{escape level} of $P$. 
\end{definition} 
The importance of escape points is explained by the following auxiliary result. Namely, if $P$ is an escape point for $S$, then we can bound quantities of the form $h(f(P))/\deg(f)$ from below (in a nontrivial way). This may be viewed as analogous to $P$ having positive canonical height when iterating a single function. However, this is not a perfect analogy, since canonical heights do not exist in general for right iteration; see Example \ref{eg:left-right difference} above. 
\begin{lemma}\label{lem:escapept} Let $S$ be a finite set of endomorphisms of $\mathbb{P}^N(\overline{\mathbb{Q}})$ all of degree at least two and let $P$ be an escape point for $S$ with escape level $r\geq0$. Then there exist positive constants $B_{S,P,1}$ and $B_{S,P,2}$ (depending on $P$) such that   
\[0<B_{S,P,1}\leq\frac{h(f(P))}{\deg(f)}\leq B_{S,P,2}\]
for all $f\in M_{S,n}$ with $r\leq n$. 
\end{lemma}
\begin{proof} The upper bound on $h(f(P))/\deg(f)$ follows directly from Lemma \ref{lem:tate} applied to the map $\rho=f$ and the point $Q=P$. For the lower bound, let $r\geq0$ be the escape level of $P$ and let $f\in M_{S,n}$ for some $n\geq r$. Then we can write $f=j\circ g$ for some $j\in M_{S,n-r}$ and some $g\in M_{S,r}$. Then Lemma \ref{lem:tate} applied to the map $\rho=j$ and the point $Q=g(P)$ implies that \vspace{.1cm} 
\begin{equation}\label{lbdespace}
\begin{split} 
\frac{h(f(P))}{\deg(f)}=\frac{h(j(g(P)))}{\deg(j)\deg(g)}&\geq\frac{1}{\deg(g)}\big(h(g(P))-B_S\big)\\[8pt]
&\geq \frac{1}{\displaystyle{\max_{g\in M_{S,r}}\{\deg{g}\}}}\cdot\min_{g\in M_{S,r}}\big\{h(g(P))-B_S\big\}
\end{split} 
\end{equation}
However, since $S$ is a finite set and $r$ is fixed, the degree of $g\in M_{S,r}$ is absolutely bounded. Likewise, since $P$ is an escape point for $S$, the quantity $h(g(P))-B_S$ is positive for all $g\in M_{S,r}$. Therefore, the minimum on the right hand side of (\ref{lbdespace}) is positive, since it is the minimum value of a finite set of positive numbers. In particular, there is a positive constant $B_{S,P,2}$, depending only on $S$ and $P$, such that $h(f(P))/\deg(f)>B_{2,P}$ as claimed.  
\end{proof} 
With Lemma \ref{lem:escapept} in place, we are ready to prove an improved version of Theorem \ref{thm:iteratedlogs} from the Introduction for escape points.  
\begin{theorem}\label{thm:escapepoints} Let $S$ be a finite set of endomorphisms of $\mathbb{P}^N(\overline{\mathbb{Q}})$ all of degree at least two, and let $\nu$ be a discrete probability measure on $S$. Then the following statements hold: \vspace{.25cm}
\begin{enumerate} 
\item[\textup{(1)}] The dynamical degree is given by $\displaystyle{\delta_{S,\nu}=\prod_{\phi\in S}\deg(\phi)^{\nu(\phi)}}$. \vspace{.3cm} 
\item[\textup{(2)}] For $\bar{\nu}$-almost every $\gamma\in\Phi_S$, the limits \vspace{.1cm}  
\[\lim_{n\rightarrow\infty}h(\gamma_n^-(P))^{1/n}=\delta_{S,\nu}=\lim_{n\rightarrow\infty}h(\gamma_n^+(P))^{1/n}\vspace{.15cm}\]
hold (simultaneously) for all escape points $P$ for $S$. \vspace{.4cm}  
\item[\textup{(3)}] If the variance $\sigma^2$ of $\log(\deg(\phi))$ is nonzero, then for $\bar{\nu}$-almost every $\gamma\in\Phi_S$, \vspace{.1cm}
\[\limsup_{n\rightarrow\infty}\frac{\log\bigg(\mathlarger{\frac{h(\gamma_n^{\pm}(P))}{\delta_{S,\nu}^n}}\bigg)}{\sigma_{S,\nu}\sqrt{2n\log\log n}}=1=\limsup_{n\rightarrow\infty}\frac{\log\bigg(\mathlarger{\frac{\delta_{S,\nu}^n}{h(\gamma_n^{\pm}(P))}}\bigg)}{\sigma_{S,\nu}\sqrt{2n\log\log n}},\vspace{.3cm}\]  
hold (simultaneously) for all escape points $P$ for $S$. \vspace{.1cm}
\end{enumerate}  
\end{theorem} 
\begin{remark} Note that if $h(P)>B_S$, then $P$ is an escape point for $S$ of level $r=0$. In particular, Theorem \ref{thm:escapepoints} implies Theorem \ref{thm:iteratedlogs} from the Introduction.  
\end{remark} 
\begin{proof}[(Proof of Theorem \ref{thm:iteratedlogs})] 
For statement (1), consider $f_1:\Phi_S\rightarrow\mathbb{R}$ given by:  
\[f_1(\gamma)=\log\deg(\theta_1)\qquad \text{for\;\,$\gamma=(\theta_i)_{i=1}^\infty\in\Phi_S$}.\]
Then Birkhoff's Ergodic Theorem \ref{birk} and Lemma \ref{shift} together imply that 
\[\lim_{n\rightarrow\infty} \frac{1}{n}\sum_{j=0}^{n-1} f_1\circ T^j(\gamma)=\mathbb{E}_{\bar{\nu}}[f_1].\]
for almost every $\gamma\in\Phi_S$; here $T:\Phi_S\rightarrow\Phi_S$ is the shift map. On the other hand, since 
\[\deg(F\circ G)=\deg(F)\cdot\deg(G)=\deg(G)\cdot\deg(F)=\deg(G\circ F)\]
for all endomorphisms $F$ and $G$ on $\mathbb{P}^N$, we have that 
\[\log\deg(\gamma^{\pm}_n)^{1/n}=\frac{1}{n}\sum_{j=0}^{n-1} f_1\circ T^j(\gamma).\]
In particular, $\delta_{S,\nu}=\displaystyle{\lim_{n\rightarrow\infty}\deg(\gamma_n^{\pm})^{1/n}}=\exp\big(\mathbb{E}_{\bar{\nu}}[f_1]\big)$ almost surely. However, $f_1:\Phi_S\rightarrow\mathbb{R}$ factors through $S$, so that \cite[Theorem 10.4]{ProbabilityText} implies that    
\[\delta_{S,\nu}=\exp\big(\mathbb{E}_{\bar{\nu}}[f_1]\big)=\exp\big(\mathbb{E}_{\nu}[\log\deg(\phi)]\big)=\exp\Big(\sum_{\phi\in S}\log\deg(\phi)\nu(\phi)\Big)=\prod_{\phi\in S}\deg(\phi)^{\nu(\phi)}\]
as claimed. For statement (2), let $\gamma\in\Phi_S$ be such that $\lim\deg(\gamma_n^{\pm})^{1/n}=\delta_{S,\nu}$, true of almost every $\gamma\in \Phi_S$, and let $P$ be an escape point for $S$. Then Lemma \ref{lem:escapept} implies that there are positive constants $B_{S,P,1}$ and $B_{S,P,2}$ such that  
\[\qquad B_{S,P,1}\cdot\deg(\gamma_n^{\pm})<h(\gamma_n^{\pm}(P))<B_{S,P,2}\cdot\deg(\gamma_n^{\pm}),\qquad\text{for all $\gamma\in\Phi_S$ and all $n\geq r$;}\]
here $r$ is the escape level of $P$. Therefore, taking $n$th roots of both sides and letting $n$ tend to infinity, we see that 
\[\delta_{S,\nu}=\lim_{n\rightarrow\infty}B_{S,P,1}^{1/n}\cdot\lim_{n\rightarrow\infty}\deg(\gamma_n^{\pm})^{1/n}\leq h(\gamma_n^{\pm}(P))^{1/n}\leq\lim_{n\rightarrow\infty}B_{S,P,2}^{1/n}\cdot\lim_{n\rightarrow\infty}\deg(\gamma_n^{\pm})^{1/n}=\delta_{S,\nu}.\]
Hence, for almost every $\gamma\in\Phi_S$ the limits 
\[\lim_{n\rightarrow\infty}h(\gamma_n^{\pm}(P))^{1/n}=\delta_{S,\nu}\]
hold (simultaneously) for all escape points $P$ for $S$ as claimed. For statement (3), suppose that $P$ is an escape point for $S$ and that the variance $\sigma^2$ of the random variable $\log\deg(\cdot): S\rightarrow\mathbb{R}$ is nonzero; here, $\sigma^2$ is given explicitly by  
\[\sigma^2=\sum_{\phi\in S} \big(\log\deg(\phi)-\log(\delta_{S,\nu})\big)^2\nu(\phi).\]
Then it follows from Lemma \ref{lem:escapept} that \vspace{.1cm} 
\begin{equation}\label{escapept:loght}
\lim_{n\rightarrow\infty}\frac{\log\bigg(\mathlarger{\frac{h(\gamma_n^{\pm}(P))}{\deg(\gamma_n^{\pm})}}\bigg)}{\sigma\sqrt{2n\log\log n}}=0=\lim_{n\rightarrow\infty}\frac{\log\bigg(\mathlarger{\frac{\deg(\gamma_n^{\pm})}{h(\gamma_n^{\pm}(P))}}\bigg)}{\sigma\sqrt{2n\log\log n}}\qquad \text{for all $\gamma\in\Phi_{S}$;} \vspace{.15cm}
\end{equation} 
here we simply use that the quantities $\log\frac{h(\gamma_n^{\pm}(P))}{\deg(\gamma_n^{\pm})}$ are bounded independently of $n\geq r$ by Lemma \ref{lem:escapept}. On the other hand, consider the i.i.d random variables $Y_n:\Phi_S\rightarrow\mathbb{R}$ given by 
\[Y_n(\gamma)=\frac{1}{\sigma}\big(\log\deg(\theta_n)-\log\delta\big),\qquad\text{for $\gamma=(\theta_i)_{i\geq1}\in\Phi_S$;}\]
In particular, each $Y_n$ has mean $0$ and unit variance. Therefore, the Hartman-Wintner Law of the Iterated Logarithms (Theorem \ref{thm:lawiterlogs}) for the simple random walk $S_n=Y_1+\dots +Y_n$  implies that \vspace{.05cm} 
\begin{equation}\label{escapept:lawiterlog} 
\limsup_{n\rightarrow\infty}\frac{\log\bigg(\mathlarger{\frac{\deg(\gamma_n^{\pm})}{\delta_{S,\nu}^n}}\bigg)}{\sigma\sqrt{2n\log\log n}}=1=\limsup_{n\rightarrow\infty}\frac{\log\bigg(\mathlarger{\frac{\delta_{S,\nu}^n}{\deg(\gamma_n^{\pm})}}\bigg)}{\sigma\sqrt{2n\log\log n}} \qquad \text{($\bar{\nu}$-almost surely).}
\end{equation}  
Hence, the conclusions in both (\ref{escapept:loght}) and (\ref{escapept:lawiterlog}) hold for almost every $\gamma\in\Phi_S$. Therefore, \vspace{.1cm}
 \begin{equation*}
\begin{split}
1&=\limsup_{n\rightarrow\infty}\frac{\log\bigg(\mathlarger{\frac{\deg(\gamma_n^{\pm})}{\delta_{S,\nu}^n}}\bigg)}{\sigma\sqrt{2n\log\log n}}+\lim_{n\rightarrow\infty}\frac{\log\bigg(\mathlarger{\frac{h(\gamma_n^{\pm}(P))}{\deg(\gamma_n^{\pm})}}\bigg)}{\sigma\sqrt{2n\log\log n}}\\[10pt]
&=\limsup_{n\rightarrow\infty}\frac{\log\bigg(\mathlarger{\frac{\deg(\gamma_n^{\pm})}{\delta_{S,\nu}^n}}\bigg)+\log\bigg(\mathlarger{\frac{h(\gamma_n^{\pm}(P))}{\deg(\gamma_n^{\pm})}}\bigg)}{\sigma\sqrt{2n\log\log n}}=\limsup_{n\rightarrow\infty}\frac{\log\bigg(\mathlarger{\frac{h(\gamma_n^{\pm}(P))}{\delta_{S,\nu}^n}}\bigg)}{\sigma\sqrt{2n\log\log n}}\\[5pt]
\end{split}  
\end{equation*}
for almost every $\gamma\in\Phi_S$. Likewise, (\ref{escapept:loght}) and (\ref{escapept:lawiterlog}) imply that \vspace{.3cm}  
\begin{equation*}
\begin{split}
1&=\limsup_{n\rightarrow\infty}\frac{\log\bigg(\mathlarger{\frac{\delta_{S,\nu}^n}{\deg(\gamma_n^{\pm})}}\bigg)}{\sigma\sqrt{2n\log\log n}}+ \lim_{n\rightarrow\infty}\frac{\log\bigg(\mathlarger{\frac{\deg(\gamma_n^{\pm})}{h(\gamma_n^{\pm}(P))}}\bigg)}{\sigma\sqrt{2n\log\log n}}\\[10pt] 
&=\limsup_{n\rightarrow\infty}\frac{\log\bigg(\mathlarger{\frac{\delta_{S,\nu}^n}{\deg(\gamma_n^{\pm})}}\bigg)+\log\bigg(\mathlarger{\frac{\deg(\gamma_n^{\pm})}{h(\gamma_n^{\pm}(P))}}\bigg)}{\sigma\sqrt{2n\log\log n}}=\limsup_{n\rightarrow\infty}\frac{\log\bigg(\mathlarger{\frac{\delta_{S,\nu}^n}{h(\gamma_n^{\pm}(P))}}\bigg)}{\sigma\sqrt{2n\log\log n}}\\[8pt] 
\end{split}  
\end{equation*}
holds for $\bar{\nu}$-almost every $\gamma\in\Phi_S$, whenever $P$ is an escape point for $S$ and $\sigma^2$ is nonzero. This complete the proof of Theorem \ref{thm:escapepoints}.   
\end{proof} 
As an application of Theorem \ref{thm:escapepoints}, we can prove an asymptotic formula for the number of points in generic left and right orbits.  
\begin{proof}[(Proof of Corollary \ref{cor:escapeptshtbds})] We mostly follow the proof of \cite[Proposition 3]{KawaguchiSilverman}. However, there is an added step, which allows us to pass from superscripts of $\gamma_{n}^{\pm}(P)$ to points in orbits $Q\in\Orb_\gamma^{\pm}(P)$; see Lemma \ref{lem:n'stopoints} below. Let $P$ be an escape point for $S$ and let $\gamma\in\Phi_S$ be such that $\lim h(\gamma_n^{\pm}(P))^{1/n}=\delta_{S,n}$, true of almost every $\gamma$ by Theorem \ref{thm:escapepoints}. Then for every $\epsilon>0$ there is an integer $n_0=n_0(\epsilon,\gamma)$ so that 
\[(1-\epsilon)\delta_{S,\nu}\leq h(\gamma_n^{\pm}(P))^{1/n}\leq(1+\epsilon)\delta_{S,\nu}\]
for all $n\geq n_0$; here you choose $n_0$ to be max of the corresponding $n_0(\epsilon,\gamma,-)$ and $n_0(\epsilon,\gamma,+)$. In particular, it follows that \vspace{.1cm} 
\begin{equation}\label{basiccount1}
\begin{split} 
\{n\geq n_0\,:\,(1+\epsilon)\delta_{S,\nu}\leq B^{1/n}\}&\subset\{n\geq n_0\,:\,h(\gamma_n^{\pm}(P))\leq B\} \\[2pt] 
&\text{and}\\[2pt] 
\{n\geq n_0\,:\,h(\gamma_n^{\pm}(P))\leq B\}&\subset\{n\geq n_0\,:\,(1-\epsilon)\delta_{S,\nu}\leq B^{1/n}\}. 
\end{split} 
\end{equation}
Therefore, after counting the number elements in the sets in (\ref{basiccount1}), we see that 
\begin{equation*}
\begin{split} 
\frac{\log(B)}{\log((1+\epsilon)\delta_{S,\nu})}-n_0&\leq\#\{n\geq0\,:\,h(\gamma_n^{\pm}(P))\leq B\}\\[2pt] 
&\text{and}\\[2pt] 
\#\{n\geq0\,:\,h(\gamma_n^{\pm}(P))\leq B\}&\leq\frac{\log(B)}{\log((1-\epsilon)\delta_{S,\nu})}+n_0+1. \\[3pt] 
\end{split} 
\end{equation*} 
Hence, dividing by $\log(B)$ and letting $B$ tend to infinity gives \vspace{.1cm} 
\begin{equation*}
\begin{split} 
\frac{1}{\log((1+\epsilon)\delta_{S,\nu})}&\leq\liminf_{B\rightarrow\infty}\frac{\#\{n\geq0\,:\,h(\gamma_n^{\pm}(P))\leq B\}}{\log(B)}\\[3pt] 
&\text{and} \\[3pt] 
\limsup_{B\rightarrow\infty}\frac{\#\{n\geq0\,:\,h(\gamma_n^{\pm}(P))\leq B\}}{\log(B)}&\leq \frac{1}{\log((1-\epsilon)\delta_{S,\nu})}. \\[4pt]
\end{split}  
\end{equation*}  
In particular, since $\epsilon$ was arbitrary, we deduce that   \vspace{.25cm}
\begin{equation}\label{count:n's} 
\lim_{B\rightarrow\infty}\frac{\#\{n\geq0\,:\,h(\gamma_n^{-}(P))\leq B\}}{\log(B)}=\frac{1}{\log(\delta_{S,\nu})}=\lim_{B\rightarrow\infty}\frac{\#\{n\geq0\,:\,h(\gamma_n^{+}(P))\leq B\}}{\log(B)} \vspace{.2cm} 
\end{equation}  
hold (simultaneously) for almost every $\gamma\in\Phi_S$. From here, we pass from superscripts $n$ to points in orbits by the following lemma; however, we must assume that the initial point $P$ has height at least $3B_S$ (instead of $B_S$). In particular, we deduce from (\ref{count:n's}) and Lemma \ref{lem:n'stopoints} below, that for almost every $\gamma\in\Phi_S$ the limits \vspace{.15cm} 
\[\lim_{B\rightarrow\infty}\frac{\#\{Q\in\Orb_\gamma^-(P)\,:\,h(Q)\leq B\}}{\log(B)}=\frac{1}{\log\delta_{S,\nu}}=\lim_{B\rightarrow\infty}\frac{\#\{W\in\Orb_\gamma^+(P)\,:\,h(W)\leq B\}}{\log(B)} \vspace{.15cm}\] 
hold (simultaneously) for all $P$ with $h(P)>3B_S$ as claimed.    
\end{proof}   
\begin{lemma}\label{lem:n'stopoints} Let $S$ be a height controlled set of endomorphisms of $\mathbb{P}^N(\overline{\mathbb{Q}})$. If $h(P)>3B_S$, then $\gamma_n^{-}(P)\neq\gamma_m^{-}(P)$ and $\gamma_n^{+}(P)\neq\gamma_m^{+}(P)$ for all $n\neq m$ and all $\gamma\in\Phi_S$. 
\end{lemma} 
\begin{proof} Suppose that $n>m$ and that $\gamma_n^{\pm}(P)=\gamma_m^{\pm}(P)$. In particular, $h(\gamma_n^{\pm}(P))=h(\gamma_m^{\pm}(P))$. Then Lemma \ref{lem:tate} applied separately to $f=\gamma_n^{\pm}$ and then to $f=\gamma_m^{\pm}$ implies that 
\[\deg(\gamma_n^{\pm})\cdot(h(P)-B_S)\leq h(\gamma_n^{\pm}(P))=h(\gamma_m^{\pm}(P))\leq \deg(\gamma_m^{\pm})\cdot(h(P)+B_S).\]
Rearranging terms, we deduce that 
\begin{equation}\label{distinctorbits}
\frac{\deg(\gamma_n^{\pm})}{\deg(\gamma_m^{\pm})}\leq \frac{(h(P)+B_S)}{(h(P)-B_S)}.
\end{equation} 
However, $n>m$ so that $\gamma_n^-=g_1\circ\gamma_m^-$ and $\gamma_n^+=\gamma_m^+\circ g_2$ for some $g_1,g_2\in M_{S,n-m}$. Moreover, $S$ is height controlled, so that $\deg(g_i)\geq2$. Furthermore, $\deg(\gamma_n^-)=\deg(g_1)\cdot\deg(\gamma_m^-)$ and $\deg(\gamma_n^+)=\deg(g_2)\cdot\deg(\gamma_m^+)$. Combining these facts with (\ref{distinctorbits}), we deduce that  
\[2\leq \frac{(h(P)+B_S)}{(h(P)-B_S)}.\]
However, this statement immediately implies that $h(P)\leq 3B_S$, and the result follows.   
\end{proof} 
Since we are particularly interested in arithmetic aspects of right orbits for their relation to dynamical Galois groups (see Section \ref{sec:Galois}), we give a more explicit version of the height bounds in Theorem \ref{thm:escapepoints} for finite sets of unicritical maps. 
\begin{remark} If one is interested in trying to generalize known primitive prime divisor results to right iteration, especially those which are useful for understanding dynamical Galois groups \cite{Tucker,AvgZig, Riccati}, then one likely needs (among other things) a fairly refined understanding of the growth rates of heights in right orbits.   
\end{remark} 
\begin{corollary}{\label{cor:unicritescape}} Let $S=\{x^{d_1}+c_1, \dots, x^{d_s}+c_s\}$ for some $d_i\geq2$ and some $c_i\in\mathbb{Z}\mysetminus\{0\}$. Furthermore, assume that $P\in\mathbb{Q}$ satisfies 
\[ h((P^{d_i}+c_i)^{d_j}+c_j)\geq\max_i\{\log|2c|\}\;\;\;\;\; \text{for all $i,j$.}\] 
Then $P$ is an escape point for $S$. Therefore, if $S$ is equipped with the uniform measure and the $d_i$ are not all identical, then for all $\epsilon>0$ and almost every $\gamma\in\Phi_S$, there exists $N_{\gamma,P,\epsilon}$ such that \vspace{.15cm} 
\[(d_1d_2\dots d_s)^{\frac{n-(1+\epsilon)\log_{\delta}(e)\sigma\sqrt{2n\log\log n}}{s}}\leq h(\gamma_n^+(P))\leq (d_1d_2\dots d_s)^{\frac{n+(1+\epsilon)\log_\delta(e)\sigma\sqrt{2n\log\log n}}{s}} \vspace{.15cm}\]
for all $n\geq N_{\gamma,P,\epsilon}$.    
\end{corollary} 
\begin{remark}{\label{rmk:escape}} In particular, if $|c_i^{d_j}+c_j|\geq2\max\{|c_i|\}$, then $0$ is an escape point for $S$ and the height bounds in Corollary \ref{cor:unicritescape} hold for $P=0$; for an application to Galois theory, see Corollary \ref{cor:Galoisescp}. We also note that in practice, the condition on $P$ in Corollary \ref{cor:escapeptshtbds} holds for every rational point (for many sets $S$).      
\end{remark} 
\begin{proof} Let $\phi(x)=x^d+c$ for some $d\geq2$ and some $c\in\mathbb{Z}\mysetminus\{0\}$. Then it is straightforward to prove that  
\[|h(\phi(P))-dh(P)|\leq\log|2c|\;\;\;\;\;\;\;\text{for all $P\in\mathbb{Q}$;}\] 
see \cite[Lemma 12]{Ingram}. In particular, the set $S$ as in Corollary \ref{cor:unicritescape} is height controlled with height constants $C_S=\max\{\log|2c|\}$ and $d_S\geq2$. Moreover, the condition on $P$ implies that $P$ is an escape point for $S$ with escape level $r=2$; see Definition \ref{def:escapepts} above. The claim then follows from Theorem \ref{thm:iteratedlogs} part (3) and the fact that the dynamical degree $\delta_{S,\nu}=\prod d_i^{1/s}$ is a geometric mean of the degrees of the maps in $S$; see Theorem \ref{thm:iteratedlogs} part (1).  
\end{proof}
We now move on to study right iteration more carefully for more general initial points, including some points of small height. 
\begin{remark} For left iteration this analysis is accomplished by using canonical heights. In particular, several of our results on arithmetic and dynamical degrees above (for left iteration) hold for so called almost surely wandering points; see \cite[Theorem 1.5]{Me:dyndeg}. 
\end{remark} Here, the key assumption we make on the initial point $P$ is that it have infinite total orbit, i.e., the action of the entire monoid generated by $S$ on $P$ gives an infinite set; see (\ref{eq:totalorbit}) above. In particular, this condition is weaker than the assumption that $P$ be an escape point for $S$. In this case (and among other things), we prove that $\limsup h(\gamma_n^+(P))=\delta_{S,\nu}$ almost surely; compare to Theorems \ref{thm:rationalmaps} and \ref{thm:lawiterlogs} above. Moreover, this result holds for infinite height controlled sets of endomorphisms as well. For a statement of the following result, see Theorem \ref{thm:zero-one} from the Introduction.     
\begin{proof}[(Proof of Theorem \ref{thm:zero-one})] For statement (1), let $P\in\mathbb{P}^N(\overline{\mathbb{Q}})$ be any point. Note that if $P$ is fixed and the sequence $\gamma\in\Phi_S$ is allowed to vary, then Lemma \ref{lem:tate} implies that the height-degree quotient sequence $h(\gamma^+_n(P))/\deg(\gamma_n^+)$ is bounded by $h(P)\pm B_S$. Therefore, both 
\[\displaystyle{\liminf_{n\rightarrow\infty}\frac{h(\gamma_n^+(P))}{\deg(\gamma_n^+)}}\;\;\; \text{and}\;\;\;\displaystyle{\limsup_{n\rightarrow\infty}\frac{h(\gamma_n^+(P))}{\deg(\gamma_n^+)}}\] 
exist and are $h(P)+O(1)$ for all $P\in\mathbb{P}^N(\overline{\mathbb{Q}})$ and all $\gamma\in\Phi_S$.

For statement (2), suppose that all of the maps in $S$ are defined over a fixed number field $K$. Moreover, we assume (without loss of generality) that the initial point $P\in\mathbb{P}^N(K)$. In particular, Northcott's Theorem (over $K$) implies that if $\Orb_S(P)$ is infinite, then there exists $g\in M_S$ such that $h(g(P))>B_S$. On the other hand, the Infinite Monkey Theorem, a simple consequence of Borel-Cantelli \cite[pp. 96-100]{InfiniteMonkey}, implies that 
\[\gamma_n^+=f_{\gamma,n}\circ g \;\;\text{for some $f_{\gamma,n}\in M_S$ and infinitely many $n$}\] 
for almost every $\gamma\in\Phi_S$; that is, with probability $1$ the infinite sequence $\gamma$ contains the finite substring $g$ infinitely many times. In particular, for such $\gamma$ and $n$, the bound in Lemma \ref{lem:tate} applied to $Q=g(P)$ and $\rho=f_{\gamma,n}$ implies that  
\begin{equation}\label{bdas} 
\frac{h(\gamma^+_n(P))}{\deg(\gamma^+_n)}=\frac{h(f_{\gamma,n}(g(P)))}{\deg(f_{\gamma,n})\deg(g)}\geq\frac{1}{\deg(g)}\big(h(g(P))-B_S\big)>0.
\end{equation}
It follows that the limsup of the quotient $h(\gamma^+_n(P))/\deg(\gamma^+_n)$ must be strictly positive for almost every $\gamma\in\Phi_S$. Conversely, if the limsup of $h(\gamma^+_n(P))/\deg(\gamma^+_n)>0$ is positive for a single $\gamma$ (in particular, if it's true almost surely), then the right orbit $\Orb_\gamma^+(P)$ must be infinite. Therefore, the total orbit $\Orb_S(P)$ is infinite as well.  

Finally, statement (3). Let $P$ be any initial point and let $\gamma$ be any sequence. We first show that $\lim h(\gamma_n^+(P))^{1/n}\leq\delta_{S,\nu}$ almost surely. Note that for finite sets, this is known by Theorem \ref{thm:rationalmaps}; however, we wish to allow suitable infinite sets. To do this (and to ease notation), let
\begin{equation}\label{upperht}
\bar{h}_\gamma^+(P)=\limsup_{n\rightarrow\infty}\frac{h(\gamma_n^+(P))}{\deg(\gamma_n^+)}. 
\end{equation}   
Then by definition of $\limsup$ and Theorem \ref{thm:zero-one} part (1), we know that for all $\epsilon>0$ there is an $N_{P,\gamma,\epsilon}$ such that
\[\frac{h(\gamma^+_n(P))}{\deg(\gamma^+_n)}\leq(1+\epsilon)\bar{h}^+_\gamma(P)\]
holds for all $n>N_{P,\gamma,\epsilon}$. In particular,
\begin{equation}\label{arithdegbd1} 
h(\gamma^+_n(P))^{1/n}\leq(1+\epsilon)^{1/n}\,\bar{h}^+_\gamma(P)^{1/n}\,\deg(\gamma^+_n)^{1/n}
\end{equation} 
holds for such $n$. On the other hand, if $\Orb_S(P)$ is infinite, then $\bar{h}^+_\gamma(P)$ is positive almost surely by part (2) above. Likewise, if $\mathbb{E}_\nu[\log\deg(\phi)]$ exists, then Birkhoff's Ergodic Theorem \ref{birkhoff} (and an identical argument given for Theorem \ref{thm:iteratedlogs} part (1) above) implies that \vspace{.1cm}  
\[\displaystyle{\lim_{n\rightarrow\infty}\deg(\gamma^+_n)^{1/n}}=\displaystyle{\lim_{n\rightarrow\infty}\deg(\gamma^-_n)^{1/n}}=\delta_{S,\nu}=\exp\big(\mathbb{E}_\nu[\log\deg(\phi)]\big)\qquad\;\;\;\;\text{(almost surely)};\vspace{.1cm}\] 
alternatively, we can quote \cite[Theorem 1.5]{Me:dyndeg}. Therefore, if both $\Orb_S(P)$ is infinite and the quantity $\mathbb{E}_\nu[\log\deg(\phi)]$ exists, then the bound in (\ref{arithdegbd1}) implies that 
\[\limsup_{n\rightarrow\infty} h(\gamma_n^+(P))^{1/n}\leq\delta_{S,\nu}\] 
is true for almost every $\gamma\in\Phi_S$. For the reverse inequality, suppose that $\Orb_S(P)$ is infinite and the quantity $\mathbb{E}_\nu[\log\deg(\phi)]$ exists. Then by definition of $\bar{h}_\gamma^+(P)$, for all $0<\epsilon<1$ there exists an infinite sequence $\{n_k\}\subseteq\mathbb{N}$, depending on both $\epsilon$ and $\gamma$, such that \vspace{.1cm} 
\[\bar{h}^+_\gamma(P)(1-\epsilon)\leq \frac{h(\gamma_{n_k}^+(P))}{\deg(\gamma_{n_k}^+)}\] 
for all $n_k$. In particular, we see that \vspace{.1cm}  
\[\bar{h}_\gamma^+(P)^{1/n_k}(1-\epsilon)^{1/n_k}\deg(\gamma_{n_k}^+)^{1/n_k}\leq h(\gamma_{n_k}^+(P))^{1/n_k}.\] 
Therefore, it follows that \vspace{.05cm} 
\begin{equation}\label{ineq:reverse1}
\begin{split} 
\limsup_{n_k\rightarrow\infty}\Big(\bar{h}_\gamma^+(P)^{1/n_k}(1-\epsilon)^{1/n_k}\deg(\gamma_{n_k}^+)^{1/n_k}\Big) &\leq \limsup_{n_k\rightarrow\infty} h(\gamma_{n_k}^+(P))^{1/n_k} \\[3pt] 
&\leq \limsup_{n\rightarrow\infty} h(\gamma_{n}^+(P))^{1/n}.
\end{split} 
\end{equation} 
On the other hand, $\bar{h}_\gamma^+(P)$ is almost surely positive by part (2) of Theorem \ref{thm:zero-one} above. Hence,  
\begin{equation}\label{ineq:reverse2} 
\begin{split} 
\lim_{n_k\rightarrow\infty} \bar{h}_\gamma^+(P)^{1/n_k}=1&=\lim_{n_k\rightarrow\infty}(1-\epsilon)^{1/n_k}\\[3pt] 
\;\;\;\;\;\;\;\;\;\;\text{and}&\\[3pt]  
\lim_{n_k\rightarrow\infty}\deg(\gamma_{n_k}^+)^{1/n_k}=& \lim_{n\rightarrow\infty}\deg(\gamma_{n}^+)^{1/n}=\delta_{S,\nu}
\end{split} 
\end{equation}
almost surely. Therefore, (\ref{ineq:reverse1}) and (\ref{ineq:reverse2}) together imply that 
\[\delta_{S,\nu}\leq \limsup_{n\rightarrow\infty} h(\gamma_{n}^+(P))^{1/n}\] 
holds for almost every $\gamma\in\Phi_S$ as claimed. 
\end{proof}
\begin{remark} The liminf and limsup in Theorem \ref{thm:zero-one} part (1) can be distinct for initial points $P$ of small height, even if the total orbit of $P$ is infinite; see Example \ref{eg:left-right difference} above.  
\end{remark} 
We note the following consequence of Theorem \ref{thm:zero-one}, a sort of zero-one law for finite orbit points. In particular, the analogous statement fails for left iteration; see \cite[Example 1.10]{Me:dyndeg}.   
\begin{corollary} Let $S$ be a height controlled set of endomorphisms of $\mathbb{P}^N(\overline{\mathbb{Q}})$ all defined over a fixed number field $K$ and let $\nu$ be a discrete probability measure on $S$. Then for all $P\in\mathbb{P}^N(\overline{\mathbb{Q}})$, the probability that $\Orb_\gamma^+(P)$ is finite is either $0$ or $1$.  
\end{corollary}
\begin{proof} 
Suppose that $\Orb_S(P)$ is finite. Then $\Orb^+_\gamma(P)\subseteq\Orb_S(P)$ is finite for all $\gamma\in\Phi_S$. In particular, the probability that $\Orb^+_\gamma(P)$ is finite is $1$. On the other hand, if $\Orb_S(P)$ is infinite, then part (2) of Theorem \ref{thm:zero-one} implies that 
\[I_P=\{\gamma\in\Phi_S\::\; \bar{h}^+_{\gamma}(P)>0\}\]
has full measure in $\Phi_S$; see \ref{upperht} for a definition of $\bar{h}^+_{\gamma}(P)$. On the other hand, it is clear that 
\[\{\gamma\in\Phi_S\::\; \text{$\Orb^+_\gamma(P)$ is finite}\}\subseteq \Phi_S\mysetminus I_P.\]
Therefore, the probability that $\Orb^+_\gamma(P)$ is finite is $0$. Hence, the probability that $\Orb^+_\gamma(P)$ is finite is either $0$ or $1$ as claimed. 
\end{proof}  
As a further application of Theorem \ref{thm:zero-one}, we record the following result for sets of quadratic polynomials with integral coefficients; see \cite{IJNT} for related work on sets of quadratic polynomials with rational coefficients.  
\begin{corollary}\label{cor:quad} Let $S=\{x^2+c_1,x^2+c_2,\dots,x^2+c_s\}$ for some distinct $c_i\in\mathbb{Z}$. If $s\geq3$, then 
\[0<\limsup_{n\rightarrow\infty}\frac{h(\gamma_n^+(P))}{\deg(\gamma_n^+)}\qquad\text{(almost surely)}\] for all $P\in\mathbb{Q}$ (independent of the choice of $\nu$).   
\end{corollary} 
\begin{proof} Combine Theorem \ref{thm:zero-one} part (2) with \cite[Corollary 1.2]{IJNT}.   
\end{proof} 
Finally, we apply Theorem \ref{thm:zero-one} to the height counting problem in orbits; compare to similar results in \cite[Corollary 1.16]{Me:dyndeg} and Corollary \ref{cor:escapeptshtbds} above. However, without further conditions on the initial point $P$, we can only give lower bounds.  
\begin{corollary}{\label{cor:orbitcount}} Let $S$ be a height controlled set of endomorphisms of $\mathbb{P}^N(\overline{\mathbb{Q}})$ all defined over a fixed number field $K$ and let $\nu$ be a discrete probability measure on $S$. Moreover, suppose the following conditions hold: \vspace{.1cm} 
\begin{enumerate}
\item $\mathbb{E}_\nu[\log\deg(\phi)]$ exists. \vspace{.1cm} 
\item $\Orb_S(P)$ is infinite. \vspace{.15cm} 
\end{enumerate}
Then 
\[\frac{1}{\mathbb{E}_\nu[\log\deg(\phi)]}\leq\liminf_{B\rightarrow\infty}\;\frac{\#\{n\geq0\,:\, h(\gamma_n^+(P))\leq B\}}{\log B}\vspace{.15cm}\]  
for almost every $\gamma\in\Phi_S$.  
\end{corollary}  
We suppress the proof of Corollary \ref{cor:orbitcount} due to its similarity to Corollary \ref{cor:escapeptshtbds} above. 

\section{Height counting in total orbits}\label{sec:totalorbits}
We now turn briefly to the height counting problem for total orbits from the Introduction. However, the reader should bear in mind that the work in this section is preliminary. Nevertheless, we include it to motivate future work; for instance, we shall see how this problem relates to growth rates in semigroups and lattice point counting in various domains. As a reminder, if $P\in\mathbb{P}^N(\overline{\mathbb{Q}})$ is fixed, then our overall goal is to understand the asymptotic size of the set of points in the total orbit of $P$ of height at most $B$,    
\[\{Q\in\Orb_S(P):\, h(Q)\leq B\},\]
as $B$ grows. However, at the moment this problems seems quite difficult (since distinct functions can agree on subvarieties), and we instead study the asymptotic size of the related set of functions 
\begin{equation}\label{monoidcount}
\{f\in M_S:\, h(f(P))\leq B\},
\end{equation} in hopes that this count will shed light on the number of points in $\Orb_S(P)$ of bounded height. The basic idea, consistent with our work on orbits coming from sequences, is that the height of a point $f(P)\in \Orb_S(P)$ is roughly determined by the size of $\deg(f)$, as long as the initial point $P$ is sufficiently generic; see Lemma \ref{lem:escapept} With this in mind, to count the number of functions $f\in M_S$ with $h(f(P))\leq B$, we should in some sense simply be counting the number of $f$'s of bounded degree. Moreover, when $M_S$ is (in a nice way) generated by a set of morphisms, this problem may be tractable. 

To make this heuristic precise, we briefly discuss weighted lengths on monoids. Let $M$ be a monoid generated by a finite set $S=\{\phi_1,\dots,\phi_s\}$ and let $c=(c_1,\dots,c_s)\in\mathbb{R}^s$ be a vector of positive weights. Then we define the \emph{weighted length} $\mathit{l}_{S,c}(f)$ of any $f\in M$ as follows. First let $\Sigma(S)$ be the free monoid generated by $S$ (i.e., $\Sigma(S)$ is the set of all words in the alphabet $S$) and define $\mathit{l}_{S,c}(\phi_i)=c_i$. Then extend $\mathit{l}_{S,c}$ to any word $\sigma\in\Sigma(S)$ by setting $\mathit{l}_{S,c}(\sigma)=\mathit{l}_{S,c}(s_1)+\dots+\mathit{l}_{S,c}(s_k)$ whenever $\sigma=s_1\cdots s_k$ and $s_i\in S$. Finally, for $f\in M$ we define $\mathit{l}_{S,c}(f)$ to be 
\[\mathit{l}_{S,c}(f):=\inf\big\{\mathit{l}_{S,c}(\sigma): \sigma\in\Sigma(S)\;\text{and $\sigma$ represents $f$}\big\}.\] 
Moreover, given a notion of length, one can study the growth function $g_{S,c}:\mathbb{R}\rightarrow\mathbb{N}$ given by 
\begin{equation}\label{growth} 
g_{S,c}(B):=\#\{f\in M\,:\, \mathit{l}_{S,c}(f)\leq B\}.
\end{equation} 
In particular, the growth rate of $g_{S,c}$ may be used to encode information about the Monoid $M$ and the generating set $S$. 
\begin{remark} 
Historically, most of the work on this problem has focused on the case when $M$ is a group and each $c_i=1$ (with some additional work on the case when $c_i\in\mathbb{N}$ also); see \cite{growth1,growth2}. However, the relevant definitions make sense for $c_i\in\mathbb{R}_{>0}$ and monoids, and this is the situation that arises most naturally in our work here.  
\end{remark} 
Back to dynamics. Let $S=\{\phi_1, \dots,\phi_s\}$ be a finite set of endomorphisms on $\mathbb{P}^N$ all of degree at least $2$, let $c_i=\log\deg(\phi_i)$, and define $\mathit{l}(f):=\log\deg(f)$ for all $f\in M_S$. Then it is straightforward to check that $\mathit{l}(f)=\mathit{l}_{S,c}(f)$ independent of $S$ (the degree of a composite morphism is the product of the degrees of its components, and the degree of a function is intrinsic, i.e., does not depend on how it is written as a composition of other functions). 

Now suppose that $P\in\mathbb{P}^N(\overline{\mathbb{Q}})$ is such that $h(P)>B_S:=C_S/(d_S-1)$; here $C_S$ and $d_S$ are the constants from Definition \ref{def:htcontrolled} above. Then, Tate's telescoping Lemma \ref{lem:tate} implies that \vspace{.1cm}
\[\deg(f)(h(P)-B_S)\leq h(f(P))\leq\deg(f)(h(P)+B_S).\vspace{.1cm}\] 
Therefore, for all $B$ we have the subset relations: \vspace{.1cm}
\begin{equation}\label{subset}
\Scale[.835]{\;\,\bigg\{f\in M_S\,:\,\mathit{l}(f)\leq \log\big(\frac{B}{h(P)+B_S}\big)\bigg\}\subseteq \big\{f\in M_S:\, h(f(P))\leq B\big\}\subseteq\bigg\{f\in M_S\,:\,\mathit{l}(f)\leq \log\big(\frac{B}{h(P)-B_S}\big)\bigg\}}. \vspace{.1cm} 
\end{equation} 
In particular, (\ref{growth}) and (\ref{subset}) imply that
\begin{equation}\label{ht-wt} 
\{f\in M_S:\, h(f(P))\leq B\}\sim g_{S,c}(\log\,B)
\end{equation} 
as $B$ tends to infinity. 

As an application, we consider the case when $S$ is a free basis of the commutative monoid $M_S$ (as an example, one may take $S=\{x^{d_1}, \dots, x^{d_s}\}$ where the $d_i\in\mathbb{N}$ are multiplicatively independent). In this case, $M_S\cong \mathbb{N}^s$ with the operation of coordinate addition, and it is straightforward to check that 
\[g_{S,c}(B')=\#\{(e_1,\dots, e_s)\in\mathbb{N}^s\,:\,e_1c_1+e_2c_2+\dots+e_sc_s\leq B'\}.\] 
However, this is evidently a count of the number of lattice points in a dilate of the bounded, Jordon measurable region 
\[\Omega=\{(x_1,\dots, x_s)\in\mathbb{R}^s\,:\, 0\leq x_i\;\text{and}\; x_1c_1+\dots+x_sc_s\leq 1\}.\] 
In particular, since the volume of $\Omega$ is $(s!c_1c_2\dots c_s)^{-1}$ it follows that 
\[g_{S,c}(B')\sim (s!c_1c_2\dots c_s)^{-1} (B')^s\] 
as $B'$ tends to infinity; see, for instance, \cite[Theorem 12.2]{Pollack}. Letting $B'=\log(B)$, we deduce from (\ref{ht-wt}) that 
\[\lim_{B\rightarrow\infty}\frac{\#\Big\{f\in M_S\,:\,h\big(f(P)\big)\leq B\Big\}}{(\log B)^s}=\frac{1}{s!\cdot\prod_{i=1}^s\log\deg(\phi_i)}\]
as claimed in the Introduction. However, it seems that generically $M_S$ is a free (non-commutative) monoid, and there appears to be little (precise) information known about the growth rate function $g_{S,c}$ in this case, limiting what we can say about the dynamics. 
\begin{remark} When $M_S$ is a free non-commutative monoid (with basis $S$) and $c_i\in\mathbb{N}$, then $g_{S,c}(B)$ is a sum over restricted compositions of integers $n\leq B$; see \cite[\S2]{compositions}. In particular, one may be able to use the associated generating function to obtain an asymptotic for $g_{S,c}(B)$ in this case. However, the weights coming from dynamics are never integers (they are logs of integers). Nevertheless, since we are mainly interested in asymptotics for (\ref{monoidcount}), it is possible that the integer weight case could provide sufficient information to answer the general case.     
\end{remark}                          
\section{Galois groups generated by multiple unicritical polynomials}{\label{sec:Galois}}
We now discuss the relation between the arithmetic of right orbits and certain dynamical Galois groups. Many of the results in this  section are straightforward adaptations of analogous results for constant sequences (i.e., iterating one function); see, for instance, \cite{Jones} and \cite{Jones-Survey}. For additional work on Galois groups generated by iterating multiple maps, see \cite{Ferraguti}.    

We begin with some notation. Let $K$ be a field of characteristic $0$, and let $S$ be a set of polynomials over $K$. Then given an infinite sequence $\gamma\in\Phi_S$, we can form a tower of Galois extensions $K_{\gamma,n}:=K(\gamma_n^+)$ for $n\geq0$; here $K(\gamma_n^+)$ denotes the splitting field of the equation $\gamma_n^+(x)=0$ in a fixed algebraic closure $\overline{K}$. We note that the direction of iteration is crucial to create nested extensions: 
\[K\subseteq K_{\gamma,1}\subseteq\dots \subseteq K_{\gamma,n}\,.\]      
As in the case of iterating a single function (under some separability assumptions), the Galois group $G_{\gamma,n}:=\Gal(K_{\gamma,n}/K)$ acts naturally on the corresponding truncated \emph{preimage tree} with vertices      
\[T_{\gamma,n}:=\big\{\alpha\in\overline{K}\,:\,\gamma_m^+(\alpha)=0\; \text{for some $1\leq m\leq n$}\big\}\]
and edge relation: if $\gamma_m^+(\alpha)=0$ for some $1\leq m\leq n$ and $\gamma_m^+=\theta_1\circ\dots\circ \theta_m$, then there is an edge between $\alpha$ and $\theta_m(\alpha)$. Likewise, the inverse limit of Galois groups $\displaystyle{G_\gamma:=\lim_{\leftarrow} G_{\gamma,n}}$ acts continuously on the complete preimage tree $\displaystyle{T_\gamma=\cup_{n\geq1} T_{\gamma,n}}$ and we obtain an embedding,    
\[G_{\gamma,K}\leq \Aut(T_\gamma),\]
called the \emph{arboreal representation} of $\gamma$; see \cite[\S2]{Ferraguti} for more details. In particular, in light of our probabilistic approach in this paper and the recent finite index theorems and conjectures in \cite{Bridy-Tucker,Jones-Survey}, we pose the following question. 
\begin{question}\label{question:Galois} Let $\nu$ be a probability measure on $S$. Under what assumptions on the polynomials in $S$ can we conclude that 
\[\bar{\nu}\Big(\big\{\gamma\in\Phi_S\,:\,[\Aut(T_\gamma):G_{\gamma,K}]<\infty\big\}\Big)>0?\] 
That is, when are the arboreal representations above finite index subgroups with positive probability?  
\end{question} 
As a first step in understanding this problem, we simplify the setup substantially. Let $S$ be set of unicritical polynomials with a common critical point $c\in K$, that is
\begin{equation}\label{unicrit}
S=\big\{a(x-c)^{d}+b:\,a,b\in K, a\neq0, d\geq2\big\}. 
\end{equation}  
\begin{remark} In practice, especially given our work on heights in the previous sections, we usually restrict ourselves to finite subsets of (\ref{unicrit}). However, for completeness, we keep the Galois theory results in this section as general as possible.    
\end{remark}
In particular, if $K$ is a global field and $S$ is a set of polynomials as in (\ref{unicrit}), then we can restrict the ramification of the extensions $K_{\gamma,n}/K$ to the primes dividing elements of the \emph{critical orbits} $\Orb_\gamma^+(c)$ and the primes dividing the leading coefficients or degrees of the polynomials $\gamma_m^+$ for some $1\leq m\leq n$; compare to \cite[Lemma 2.6]{Jones}. In what follows, we use the shorthand $\ell(f)$ and $d(f)$ for the leading term and degree respectively of a polynomial $f\in K[x]$. Moreover, because this section is entirely devoted to right iteration, we (at times) drop the superscript $+$ and simply write $\gamma_n$ for $\gamma_n^+$ when convenient.  
\begin{proposition}\label{prop:discriminant} Let $S$ be a set of polynomials as in (\ref{unicrit}). Moreover, given $\gamma=(\theta_i)_{i=1}^\infty\in\Phi_S$ and $n\geq0$, let $\ell_{\gamma,n}$, $d_{\gamma,n}$, and $\Delta_{\gamma,n}$  be the leading term, the degree, and the discriminant of $\gamma_n^+$ respectively. Then 
\[\Delta_{\gamma,n}=\pm\, d(\theta_n)^{d_{\gamma,n}}\cdot\ell_{\gamma,n-1}^{\,d(\theta_n)-1}\cdot\ell(\theta_n)^{d_{\gamma,n-1}(d_{\gamma,n}-1)}\cdot\gamma_n^+(c)\cdot\Delta_{\gamma,n-1}^{d(\theta_n)}\]
for all $n\geq1$.       
\end{proposition} 
\begin{proof} We begin with a few well known facts about discriminants and resultants; see, for instance, \cite[IV \S8]{Lang}. Let $h_1, h_2, h_3 \in K[x]$ be nonconstant polynomials. Then the resultant $\Res(h_1,h_2)$ of $h_1$ and $h_2$ is given by 
\begin{equation}\label{resultant}
\Res(h_1,h_2)=\ell(h_1)^{d(h_2)}\prod_{h_1(\alpha)=0}h_2(\alpha),
\end{equation} 
where the product above is taken over roots $\alpha\in\overline{K}$ of $h_1$ with multiplicity. Then the discriminant $\Delta(h_1)$ of $h_1$ satisfies \vspace{.1cm}
\begin{equation}\label{discriminant1} 
\Res(h_1,h_1')=(-1)^{d(h_1)(d(h_1)-1)/2}\ell(h_1)\Delta(h_1).\vspace{.1cm}
\end{equation} 
In particular, it is straightforward to check that $\Res(h_1,h_2)=(-1)^{d(h_1)d(h_2)}\Res(h_2,h_1)$, that $\Res(h_1h_2,h_3)=\Res(h_1,h_3)\Res(h_2,h_3)$, and that \vspace{.1cm}
\begin{equation}\label{discriminant2}
\Res(h_1\circ h_2, h_1'\circ h_2)=\ell(h_2)^{(d(h_1)^2-d(h_1))d(h_2)}\Res(h_1,h_1')^{d(h_2)}. \vspace{.1cm}
\end{equation}
We now apply these facts to the discriminants in Proposition \ref{prop:discriminant}. Specifically, it follows from (\ref{discriminant1}) and that $\gamma_n^+=\gamma_{n-1}^+\circ\theta_n$ that \vspace{.1cm}
\begin{equation}\label{discriminant3} 
\frac{\Delta_{\gamma,n}}{\Delta_{\gamma,n-1}^{d(\theta_n)}}=\pm\frac{\ell_{\gamma,n-1}^{\,d(\theta_n)}}{\ell_{\gamma,n}}\cdot \frac{\Res(\gamma_n,\gamma_n')}{\Res(\gamma_{n-1},\gamma_{n-1}')^{d(\theta_n)}}; \vspace{.1cm} 
\end{equation} 
here we have dropped the superscript $+$ to avoid overly cumbersome notation. On the other hand, the chain rule implies that $\gamma_n'=(\gamma_{n-1}'\circ\theta_n)\cdot\theta_n'$. In particular, the standard resultant facts above together with (\ref{discriminant2}) imply that \vspace{.1cm}
\begin{equation}\label{discriminant4}
\begin{split} 
\Res(\gamma_n,\gamma_n')=&\pm \Res(\gamma_n',\gamma_n)\\[3pt] 
=&\pm\Res((\gamma_{n-1}'\circ\theta_n)\cdot\theta_n', \gamma_n)\\[3pt]
=&\pm \Res(\gamma_{n-1}'\circ\theta_n, \gamma_n)\,\Res(\theta_n', \gamma_n)\\[3pt] 
=&\pm\Res(\gamma_{n-1}'\circ\theta_n, \gamma_{n-1}\circ\theta_n)\,\Res(\theta_n', \gamma_n)\\[3pt]
=&\pm\Res(\gamma_{n-1}\circ\theta_n,\gamma_{n-1}'\circ\theta_n)\,\Res(\theta_n', \gamma_n)\\[3pt] 
=&\pm\ell(\theta_n)^{(d_{\gamma,n-1}^{\,2}-\,d_{\gamma,n-1})d(\theta_n)}\,\Res(\gamma_{n-1},\gamma_{n-1}')^{d(\theta_n)}\,\Res(\theta_n', \gamma_n)
\end{split} 
\end{equation}
Therefore, combining the expression in (\ref{discriminant3}) with the bottom line of (\ref{discriminant4}), we see that
\begin{equation}\label{discriminant5}
\frac{\Delta_{\gamma,n}}{\Delta_{\gamma,n-1}^{d(\theta_n)}}=\pm\frac{\ell_{\gamma,n-1}^{\,d(\theta_n)}}{\ell_{\gamma,n}}\cdot \ell(\theta_n)^{(d_{\gamma,n-1}^{\,2}-\,d_{\gamma,n-1})d(\theta_n)}\,\Res(\theta_n', \gamma_n).
\end{equation}  
However, using the definition of the resultant in (\ref{resultant}) and the fact that $\theta_n$ has a unique critical point $c$, we see that $\Res(\theta_n', \gamma_n)=\ell(\theta_n')^{d_{\gamma,n}}\gamma_n(c)$. Hence, (\ref{discriminant5}) may be rewritten as 
\begin{equation}\label{discriminant6}
\frac{\Delta_{\gamma,n}}{\Delta_{\gamma,n-1}^{d(\theta_n)}}=\pm\frac{\ell_{\gamma,n-1}^{\,d(\theta_n)}}{\ell_{\gamma,n}}\cdot \ell(\theta_n)^{(d_{\gamma,n-1}^{\,2}-\,d_{\gamma,n-1})d(\theta_n)}\,\ell(\theta_n')^{d_{\gamma,n}}\gamma_n(c).
\end{equation}
Hence, we need only control the relevant leading terms to complete the proof. First, since $\gamma_n=\gamma_{n-1}\circ\theta_n$, we see that $\ell_{\gamma,n}=\ell(\theta_n)^{d_{\gamma,n-1}}\ell_{\gamma,n-1}$. Moreover, $\ell(\theta_n)'=d(\theta_n)\ell(\theta_n)$. Therefore, after substituting these expressions into (\ref{discriminant6}) and simplifying like terms, we obtain the formula in Proposition \ref{prop:discriminant}.   
\end{proof} 
In particular for global fields $K$ and finite subsets $S$ of (\ref{unicrit}), we expect that $\Orb_\gamma^+(c)$ controls most of the ramification in $K_{\gamma,n}$. Specifically, suppose that $a_1,\dots, a_s$ and $d_1,\dots, d_s$ are the leading terms and degrees of a subset of the polynomials in $S$ respectively. Then by inducting on the formula in Proposition \ref{prop:discriminant} we see that if $\mathfrak{p}$ is a prime in $K$ that ramifies in $K_{\gamma,n}$, then $\mathfrak{p}\big\vert (d_1d_2\dots d_sa_1a_2\dots a_s)$ or $\mathfrak{p}\big\vert\gamma_m^+(c)$ for some $1\leq m\leq n$. Hence, if the total orbit of $c$ is finite, then Proposition \ref{prop:discriminant} provides a method for constructing many examples of finitely ramified, infinite extensions.
\begin{example}\label{eg:finitelyramified} Let $S=\{\pm{x^2}, \pm{(x^2-1)}, 2x^2-1\}$, a finite set of quadratic polynomials of the form in (\ref{unicrit}) over the rational numbers. Then we check that $\Orb_S(0)=\{0,\pm{1}\}$. In particular, it follows from Proposition \ref{prop:discriminant} that the extensions $K_{\gamma,n}=\mathbb{Q}(\gamma_n^+)$  are unramified outside of the prime $p=2$ for all $\gamma\in\Phi_S$ and all $n\geq1$. Moreover, if $\gamma=(2x^2-1, 2x^2-1, \theta_3, \dots)$, then $\gamma_n^+$ is irreducible for all $n\geq1$ by Proposition \ref{prop:irreducible} below; the point here is that after the second stage of iteration, one may choose any element of $S$. In particular, it would be interesting to compute the arboreal representations associated to such $\gamma$. The finite ramification precludes finite index in all of $\Aut(T_\gamma)$, but perhaps some subgroup of $\Aut(T_\gamma)$ furnishes the correct overgroup (for finite index with positive probability). 
\end{example} 
\begin{example}\label{eg:finitelyramified2} Likewise, for $a,c\in\mathbb{Z}$ and $a\neq0$, let $S_{a,c}=\big\{a(x-c)^2+\frac{ac-2}{a}, -a(x-c)^2+\frac{ac+2}{a}\big\}$. Then for all sequences $\gamma\in\Phi_{S_{a,c}}$ the extensions (over $\mathbb{Q}$) generated by $\gamma_n^+$ are unramified outside of the primes dividing $a$, $ac-2$, or $ac+2$.    
\end{example} 
We now move on to prove an irreducibility test for right iteration when $S$ is a set of quadratic polynomials; compare to \cite[Proposition 4.2]{Jones} and \cite[Lemma 1.2]{Stoll}. 
\begin{proposition}\label{prop:irreducible} Let $S$ be a set of quadratic polynomials of the form in (\ref{unicrit}), and let $\gamma=(\theta_i)_{i=1}^{\infty}\in\Phi_S$. If 
\begin{equation}\label{criticalorbit}
-\ell_{\gamma,1}\,\gamma_1^+(c),\,\ell_{\gamma,1}\,\gamma_2^+(c),\, \dots,\, \ell_{\gamma,1}\,\gamma_n^+(c)
\end{equation}  
are all non-squares in $K$, then $\gamma_n^+$ is irreducible over $K$.     
\end{proposition} 
\begin{proof} We proceed by induction. It is clear that if $-\ell_{\gamma,1}\,\gamma_1^+(c)$ is not a square in $K$, then $\gamma_1^+(x)=\ell_{\gamma,1}(x-c)^2+\gamma_1^+(c)$ is an irreducible quadratic polynomial over $K$. For $n\geq2$, assume that Proposition \ref{prop:irreducible} holds for $n-1$ and that the elements listed in (\ref{criticalorbit}) are all non-squares in $K$. Then $\gamma_{n-1}^+$ is irreducible by the induction hypothesis. Now let $\alpha\in\overline{K}$ be any root of $\gamma_{n-1}^+$ and let $\theta_n(x)=a(x-c)^2+b$. Moreover, assume (for a contradiction) that $\theta_n(x)-\alpha$ is reducible over $K(\alpha)$. Then $a(\alpha-b)$ must be a square in $K(\alpha)$. However, since $\gamma_{n-1}^+$ is irreducible over $K$, we see that $(1/\ell_{\gamma,n-1})\gamma_{n-1}^+(x+b)$ is a minimal polynomial of $\alpha-b$ over $K$. Hence, we have the following norm computation: 
\begin{equation*} 
\begin{split} 
N_{K(\alpha)/K}(a(\alpha-b))=a^{[K(\alpha):K]}\cdot N_{K(\alpha-b)/K}(\alpha-b)&=a^{2^{n-1}}\frac{\;(-1)^{2^{n-1}}}{\ell_{\gamma,n-1}}\,\gamma_{n-1}^+\big(0+b\big) \\[3pt] 
&=\frac{a^{2^{n-1}}}{\ell_{\gamma,n-1}}\gamma_{n-1}^+(\theta_n(c))=\frac{a^{2^{n-1}}}{\ell_{\gamma,n-1}}\gamma_n^+(c). \vspace{.05cm}
\end{split} 
\end{equation*}
Therefore (since norms of squares are squares) if $\theta_n(x)-\alpha$ is reducible over $K(\alpha)$, then $\ell_{\gamma,n-1}\gamma_n^+(c)$ is a square in $K$. On the other hand, it is straightforward to check that 
\begin{equation}\label{leadingterm}
\ell_{\gamma,m}=\ell(\theta_m)^{2^{m-1}}\,\ell(\theta_{m-1})^{2^{m-2}}\dots\,\ell(\theta_1)\;\;\; \text{for all $m\geq1$}. \vspace{.05cm}
\end{equation} 
Hence, the square class of $\ell_{\gamma,n-1}\,\gamma_n^+(c)$ in $K$ is the square class of $\ell(\theta_1)\,\gamma_n^+(c)=\ell_{\gamma,1}\,\gamma_n^+(c)$. In particular, we have contradicted our assumption that $\ell_{\gamma,1}\gamma_n^+(c)$ is a non-square in $K$. Therefore, $\theta_n(x)-\alpha$ must be an irreducible polynomial over $K(\alpha)$. Hence, Capelli's Lemma (stated directly below) applied to $g=\gamma_{n-1}^+$ and $f=\theta_n$ implies that $\gamma_n^+=\gamma_{n-1}^+\circ\theta_n$ is irreducible over $K$ as desired. 
\end{proof} 
\begin{lemma}[Capelli's Lemma] Let $K$ be a field, let $f,g\in K[x]$, and let $\alpha\in\overline{K}$ be a root of $g$. Then $g\circ f$ is irreducible over $K$ if and only if both $g$ is irreducible over $K$ and $f-\alpha$ is irreducible over $K(\alpha)$. 
\end{lemma} 
\begin{remark}\label{eg:finitelyramified+irre} Let $S=\{\pm{x^2}, \pm{(x^2-1)}, 2x^2-1\}$ be as in Example \ref{eg:finitelyramified}. Then, it is easy to check that if $\gamma$ is of the form $\gamma=(2x^2-1, 2x^2-1, \theta_3, \dots)$, then $\ell_{\gamma,1}\gamma_{n}^+(0)=2$ for all $n\geq1$. In particular, it follows from Proposition \ref{prop:irreducible} that the polynomials $\gamma_n^+$ are irreducible over the rational numbers for all $n\geq1$. Moreover, it is worth noting that the $\gamma_n^+$ (and their reciprocal polynomials for $n\geq2$) are not Eisenstein at $p=2$.  
\end{remark} 
In particular, we can use the irreducibility test in Proposition \ref{prop:irreducible} to make some progress towards Question \ref{question:Galois} for finite sets of quadratic polynomials with integral coefficients. For a reminder of the definition of escape points, see Definition \ref{def:escapepts} above.  
\begin{theorem}\label{thm:stability} Let $S=\{x^2+c_1, x^2+c_2,\dots, x^2+c_s\}$ for some distinct $c_i\in\mathbb{Z}$, and assume that $S$ has the following properties: \vspace{.05cm}
\begin{enumerate}
\item[\textup{(1)}] Some $-c_i$ is not a square in $\mathbb{Z}$. \vspace{.05cm}
\item[\textup{(2)}] $0$ is an escape points for $S$. \vspace{.05cm}
\end{enumerate}
Then for all discrete probability measures $\nu$ on $S$, we have that 
\[\bar{\nu}\Big(\big\{\gamma\in\Phi_S\,:\,\gamma_n^+\,\text{is irreducible over $\mathbb{Q}$ for all $n\geq1$}\big\}\Big)>0.\]
Equivalently, $G_{\gamma,\mathbb{Q}}$ acts transitively on $T_\gamma$ with positive probability. 
\end{theorem}
\begin{proof} Without loss of generality, we may assume that $-c_1$ is not a square in $\mathbb{Z}$. Therefore, if $\phi_1=x^2+c_1$, then it follows from the proof of \cite[Corollary 1.3]{Stoll} that $\phi_1^n(0)$ is not a square in $\mathbb{Z}$ for all $n\geq2$. In particular, $\phi_1^n$ is irreducible over $\mathbb{Q}$ for all $n\geq1$ by \cite[Corollary 1.3]{Stoll} and our assumption on $c_1$. Now consider the affine equation $E: y^2=\phi_1^2(x)$. Note that $E$ is nonsingular, since $\phi_1^2(x)$ is irreducible. In particular, there are only finitely many integer solutions $(x,y)\in\mathbb{Z}^2$ to $E$ by Siegel's Theorem. Now suppose that $\gamma\in\Phi_S$ is of the form $\gamma=(\phi_1,\phi_1, \theta_3,\dots)$ and that $\gamma_n^+(0)=y_n^2$ for some $y_n\in\mathbb{Z}$ and some $n\geq\max\{r+2\}$; here $r\geq0$ is the escape level of $0$ for $S$. Then $(x,y)=(\theta_3\circ\dots\circ\theta_n(0),y_n)$ is an integral solution to $E$. Therefore, there is a positive constant $B_E$ such that $h(\theta_3\circ\dots\circ\theta_n(0))\leq B_E$. Combining this bound with the lower bound in Lemma \ref{lem:escapept} applied to the function $f=\theta_3\circ\dots\circ\theta_n$ and the point $P=0$, we see that there is a positive constant $B_1=B_{S,0,1}$ such that 
\[0<B_{1}<\frac{h(\theta_3\circ\dots\circ\theta_n(0))}{2^{n-2}}\leq\frac{B_E}{2^{n-2}};\]
here we use that $\deg(\theta_3\circ\dots\circ\theta_n)=2^{n-2}$, since $S$ is a set of quadratic polynomials. Hence, such indices $n$ are bounded: $n\leq n_{E,0}:=\log_2(B_E/B_1)+2$. From here, define $N:=\max\{r+2,n_{E,0}\}$ and consider the sequences 
\[\Phi_{S,1,N}:=\big\{\gamma\in\Phi_S\,:\,\gamma=(\phi_1,\phi_1,\dots, \phi_1, \theta_{N+1}, \dots)\big\}.\]  
Then by definition of $N$, if $\gamma\in\Phi_{S,1,N}$, we see that $\gamma_n^+(0)$ cannot be a square in $\mathbb{Z}$ for all $n> N$. On the other hand, if $\gamma\in\Phi_{S,1,N}$, then $-\gamma_1^+(0), \gamma_2^+(0),\dots, \gamma_N^+(0)$ are all non-squares in $\mathbb{Q}$, since $\gamma_m^+(x)=\phi_1^m(x)$ for all $1\leq m\leq N$, since $\phi_1^n(0)$ is not a square in $\mathbb{Q}$ for all $n\geq2$, and since $-\phi_1(0)=-c_1$. Therefore, it follows from Proposition \ref{prop:irreducible} above, that if $\gamma\in\Phi_{S,1,N}$, then $\gamma_n^+$ is irreducible over $\mathbb{Q}$ for all $n\geq1$. However, $\bar{\nu}(\Phi_{S,1,N})=\nu(\phi_1)^N>0$ by \cite[Theorem 10.4]{ProbabilityText}, and the result follows.       
\end{proof}
In particular, we have the following immediate consequence of Theorem \ref{thm:stability} and Corollary \ref{cor:unicritescape} above; see also Remark \ref{rmk:escape}. 
\begin{corollary}\label{cor:Galoisescp}  Let $S=\{x^2+c_1, x^2+c_2,\dots, x^2+c_s\}$ for some distinct $c_i\in\mathbb{Z}$, and assume that $S$ has the following properties: \vspace{.05cm}
\begin{enumerate}
\item[\textup{(1)}] Some $-c_i$ is not a square in $\mathbb{Z}$. \vspace{.075cm}
\item[\textup{(2)}] $|c_i^2+c_j|\geq2\max\{|c_i|\}$ for all $1\leq i,j\leq s$.  \vspace{.075cm}
\end{enumerate}
Then for all discrete probability measures $\nu$ on $S$, we have that 
\[\bar{\nu}\Big(\big\{\gamma\in\Phi_S\,:\,\gamma_n^+\,\text{is irreducible over $\mathbb{Q}$ for all $n\geq1$}\big\}\Big)>0.\]
Equivalently, $G_{\gamma,\mathbb{Q}}$ acts transitively on $T_\gamma$ with positive probability. 
\end{corollary}  
We next generalize Stoll's maximality lemma \cite[Lemma 1.6]{Stoll} to sets of quadratic polynomials; see also  \cite[Lemma 3.2]{Jones}. In practice, this maximality lemma is the main tool for showing a given arboreal representation has finite index in the automorphism group of its associated preimage tree.
\begin{proposition}\label{prop:maximality} Let $S$ be a set of quadratic polynomials of the form in (\ref{unicrit}), and let $\gamma=(\theta_i)_{i=1}^{\infty}\in\Phi_S$. Assume that $n\geq1$ and that $\gamma_{n-1}^+$ is irreducible over $K$. Then the following statements are equivalent: \vspace{.15cm}
\begin{enumerate}
\item[\textup{(1)}] $[K_{\gamma,n}:K_{\gamma,n-1}]=2^{2^{n-1}}$. \vspace{.15cm}  
\item[\textup{(2)}] $\ell_{\gamma,1}\,\gamma_n^+(c)$ is not a square in $K_{\gamma,n-1}$.  
\end{enumerate} 
\end{proposition}
\begin{remark} Since $K_{\gamma,n}/K_{\gamma,n-1}$ is the compositum of at most $2^{n-1}$ quadratic extensions (one for each root of $\gamma_{n-1}^+$), we see that $[K_{\gamma,n}:K_{\gamma,n-1}]=2^{2^m}$ for some $0\leq m\leq n-1$. For this reason, when $m=n-1$ we say that the extension $K_{\gamma,n}/K_{\gamma,n-1}$ is maximal.  
\end{remark} 
\begin{proof} We begin with a few observations analogous to those in the proof of \cite[Lemma 1.6]{Stoll}. Let $\theta_n(x)=a(x-c)^2+b$, let $d=2^{n-1}$, and let $\alpha_1, \alpha_2, \dots, \alpha_{d}$ be the roots of $\gamma_{n-1}^+$ in $K_{\gamma,n-1}$. Then $K_{\gamma,n}=K_{\gamma,n-1}\big(\sqrt{a(\alpha_i-b)}:\,1\leq i\leq d\big)$ since $\pm1/a\sqrt{a(\alpha_i-b)}+c$ are roots of $\gamma_n^+$. Hence, $K_{\gamma,n}/K_{\gamma,n-1}$ is a $2$-Kummer extension and $[K_{\gamma,n}:K_{\gamma,n-1}]=2^{d-\dim(V)}$, where $V$ is the $\mathbb{F}_2$-vector space given by
\[V:=\Big\{(e_1,\dots, e_{d})\in\mathbb{F}_2^{d}\,:\,\prod_{i=1}^d(a(\alpha_i-b))^{e_i}\in (K_{\gamma,n-1})^2\Big\};\]
see \cite[VI \S8]{Lang}. On the other hand, since $G_{\gamma,n-1}:=\Gal(K_{\gamma,n-1}/K)$ permutes the roots of $\gamma_{n-1}^+$, we obtain an induced linear action of $G_{\gamma,n-1}$ on $V$. Moreover, since $G_{\gamma,n-1}$ is a $2$-group, either $\dim(V)=0$ or $V$ has a non-trivial $G_{\gamma,n-1}$-fixed vector; see \cite[I Lemma 6.3]{Lang}. However, $\gamma_{n-1}^+$ is irreducible over $K$, so that $G_{\gamma,n-1}$ acts transitively on the roots of $\gamma_{n-1}^+$. In particular, $(1,\dots,1)$ is the only possible non-trivial fixed vector. Therefore, we have deduced the following fact: either $\dim(V)=0$ or $(1,\dots,1)\in V$. However, if $(1,\dots,1)\in V$, then 
\[\prod_{i=1}^da(\alpha_i-b)=\frac{a^d\cdot(-1)^{d}}{\ell_{\gamma,n-1}}\cdot\Big(\ell_{\gamma,n-1}\prod_{i=1}^d(b-\alpha_i)\Big)=\frac{a^d}{\ell_{\gamma,n-1}}\cdot\gamma_{n-1}^+(b)=\frac{a^d}{\ell_{\gamma,n-1}}\cdot\gamma_{n}^+(c)\]
is a square in $K_{\gamma,n-1}$; here we use that $d$ is even. Moreover, (\ref{leadingterm}) implies that $\ell_{\gamma,n-1}$ is a square in $K$ times $\ell_{\gamma,1}$. In particular, $(1,\dots,1)\in V$ if and only if $\ell_{\gamma,1}\,\gamma_n^+(c)$ is a square in $K_{\gamma,n-1}$. The result easily follows.   
\end{proof}   
We combine the discriminant formula and the maximality lemma above to obtain a sufficient criterion for ensuring that a given arboreal representation (associated to a sequence of quadratic polynomials) has finite index in the automorphism group of its preimage tree. To do this, we briefly fix some notation. Let $K$ be a global field of characteristic $0$, i.e., a number field or a finite extension $K/k(t)$ of a rational function field in one variable; here $k$ has characteristic $0$. Given a finite prime $\mathfrak{p}$ of $K$, we let $v_{\mathfrak{p}}$ denote the normalized valuation on $K$ associated to $\mathfrak{p}$. Moreover, when $K$ is a number field, we let $\mathfrak{o}_K$ denote the ring of integers of $K$. When $K$ is a function field, we choose a prime $\mathfrak{p}_0$, and let $\mathfrak{o}_K$ denote the set $\{z\in K\,:\,v_{\mathfrak{p}}(z)\geq0\,\text{for all $\mathfrak{p}\neq\mathfrak{p}_0$}\}$. With these notions in place, we have the following arithmetic finite index test.  
\begin{theorem}{\label{maxtest}} Let $K$ be a global field of characteristic zero and let $S$ be a set of quadratic polynomials in $\mathfrak{o}_K[x]$ with common critical point $c\in\mathfrak{o}_K$. Assume that a sequence $\gamma=(\theta_i)_{i=1}^{\infty}\in\Phi_S$ is such that $\gamma_m^+$ is irreducible for all $m\geq1$. Moreover, assume that for all $n$ sufficiently large there exists a prime $\mathfrak{p}_{\gamma,n}$ of $K$ with the following properties: \vspace{.1cm}
\begin{enumerate}  
\item[\textup{(1)}] $v_{\mathfrak{p}_{\gamma,n}}(2)=0$. \vspace{.1cm}
\item [\textup{(2)}] $\mathfrak{p}_{\gamma,n}\neq\mathfrak{p}_0$ if $K$ is a function field. \vspace{.1cm}
\item[\textup{(3)}] $v_{\mathfrak{p}_{\gamma,n}}(\ell(\theta_m))=0$ for all $1\leq m\leq n$.\vspace{.1cm}
\item [\textup{(4)}] $v_{\mathfrak{p}_{\gamma,n}}(\gamma_m^+(c))=0$ for all $1\leq m\leq n-1$. \vspace{.1cm}
\item [\textup{(5)}] $v_{\mathfrak{p}_{\gamma,n}}(\gamma_n^+(c))\equiv1\Mod{2}$. \vspace{.1cm}
\end{enumerate}  
Then $G_{\gamma,K}$ is a finite index subgroup of $\Aut(T_\gamma)$.       
\end{theorem} 
\begin{proof} By the discriminant formula in Proposition \ref{prop:discriminant}, if $\mathfrak{p}_{\gamma,n}$ has properties $(1)$-$(4)$ above, then $\mathfrak{p}_{\gamma,n}$ must be unramified in $K_{\gamma,n-1}$. Hence properties (3) and (5) together imply that $\ell_{\gamma,1}\gamma_n^+(c)$ cannot be a square in $K_{\gamma,n-1}$. 
In particular, it follows from Proposition \ref{prop:maximality} that $K_{\gamma,n}/K_{\gamma,n-1}$ is maximal for all $n$ sufficiently large. Therefore, $G_{\gamma,K}$ is a finite index subgroup of $\Aut(T_\gamma)$ as claimed.    
\end{proof}
As a consequence of Theorem \ref{maxtest}, we construct examples over the global field $K=\mathbb{Q}(t)$ for which Question \ref{question:Galois} has an affirmative answer; here we take $\mathfrak{o}_K:=\mathbb{Q}[t]$. In what follows, $\frac{d}{dt}$ denotes the usual derivative on polynomials and $\overline{c}\in\mathbb{Z}/2\mathbb{Z}[t]$ denotes the image of $c\in\mathbb{Z}[t]$ under the ring homomorphism $\mathbb{Z}[t]\rightarrow\mathbb{Z}/2\mathbb{Z}[t]$ given by reducing coefficients.  
\begin{theorem}{\label{thm:functionfield}} Let $K=\mathbb{Q}(t)$ and let $S$ be a set of quadratic polynomials of the form $x^2+c$ such that each $c$ satisfies all of the following conditions:   
 \vspace{.1cm}
\begin{enumerate}  
\item[\textup{(1)}] $c\in\mathbb{Z}[t]$ and $\ell(c)=\pm{1}$. \vspace{.2cm}
\item [\textup{(2)}] $\deg(c)=d>0$. \vspace{.1cm}
\item[\textup{(3)}] $\displaystyle{\frac{d}{dt}}\,\overline{c}=1$.\vspace{.1cm}
\end{enumerate}  
Then $G_{\gamma,K}=\Aut(T_\gamma)$ for all $\gamma\in\Phi_S$.  
\end{theorem} 
\begin{example} In particular, the set $S=\big\{x^2+(-t^2+t+3),\; x^2+(t^2-5t)\big\}$ satisfies the hypothesis of Theorem \ref{thm:functionfield} above (with $d=2$).  
\end{example} 
\begin{remark} Although the conditions in Theorem \ref{thm:functionfield} may seem strange, their utility may be summarized as follows: conditions (1) and (3) ensure that $\gamma_n^+(0)$ is square-free and condition (2) ensures that $\deg(\gamma_n^+(0))=2^{n-1}d$. In particular, putting these facts together we deduce that $\gamma_n^+(0)$ has an irreducible factor appearing to exponent $1$, which is coprime to $\gamma_m^+(0)$ for all $1\leq m\leq n-1$ (by simple degree considerations). In particular, it follows that $K_{\gamma,n}/K_{\gamma, n-1}$ is maximal for all $n\geq1$ by Proposition \ref{prop:maximality}. 
\end{remark} 
\begin{proof} Suppose that $S$ conditions (1)-(3) of Theorem \ref{thm:functionfield} hold, and let $\gamma=(\theta_n)_{n=1}^\infty\in\Phi_S$. Then it follows easily by induction, using only that $\deg(f+g)=\max\{\deg(f),\deg(g)\}$ when $\deg(f)\neq\deg(g)$ and $\deg(f^2)=2\deg(f)$, that \vspace{.1cm} 
\begin{equation}\label{fact1}
\deg(\gamma_n^+(0))=2^{n-1}d\;\;\;\; \text{for all $n\geq1$, $\gamma\in\Phi_S$}.  
\end{equation}
Likewise, the leading term $\ell(\gamma_n^+(0))=\pm{1}$ by property (1) above. In particular, $\gamma_n^+(0)\in\mathbb{Z}[t]$ is a primitive polynomial (the gcd of its coefficients is $1$). We next show that each polynomial $\gamma_n^+(0)\in\mathbb{Q}[t]$ (a unique factorization domain) is square-free. To see this, suppose for a contradiction, that $\gamma_n^+(0)=f_n\cdot g_n^2$ for some $f_n,g_n\in\mathbb{Q}[t]$ and some non-constant $g_n$. Note that by Gauss' Lemma, we can assume that $f_n,g_n\in\mathbb{Z}[t]$; here we use that $\gamma_n^+(0)$ is primitive. In particular, after writing $\theta_1=x^2+c$ for some $c$ satisfying (1)-(3) above, we have that  
\begin{equation}\label{fact2}
f_n\cdot g_n^2=\gamma_n^+(0)=y_n^2+c
\end{equation}   
for some $y_n\in\mathbb{Z}[t]$. Moreover, since the leading term of $\gamma_n^+(0)$ is $\pm{1}$, the leading term of $g_n$ must be $\pm{1}$ also. Therefore, $\deg(g)=\deg(\overline{g})>0$, and the reduction of $g$ modulo $2$ is non-constant. On the other hand, after reducing coefficients and taking derivatives of both sides of (\ref{fact2}), we see that 
\[\Big(\frac{d}{dt}\overline{f_n}\,\Big)\cdot{\overline{g_n}}^2=\frac{d}{dt}\overline{c}=1\] 
by property (3). Hence, $\bar{g_n}$ is a unit in $\mathbb{Z}/2\mathbb{Z}[t]$. However,  this contradicts the fact that $\deg(\overline{g})>0$. Therefore, $\gamma_n^+(0)\in\mathbb{Q}[t]$ is square-free as claimed. We use this fact to analyze the relevant  Galois groups. 

Note first that since $\gamma_n^+(0)$ is non-constant and square-free in $\mathbb{Q}[t]$, Proposition \ref{prop:irreducible} implies that $\gamma_n^+$ is irreducible over $K=\mathbb{Q}(t)$ for all $n\geq1$. Likewise, if no prime $\mathfrak{p}_n$ (corresponding to an irreducible polynomial) of $\mathbb{Q}[t]$ as in Theorem \ref{maxtest} exists for $n\geq2$, then each irreducible factor $q(t)$ of $\gamma_n^+(0)$ must also divide some $\gamma_{m_q}^+(0)$ for some $1\leq m_q\leq n-1$: conditions (1)-(3) of Theorem \ref{maxtest} hold trivially, and condition (5) holds since $\gamma_n^+(0)$ is square-free. In particular, it follows that the polynomial $\gamma_n^+(0)$ divides the product $\gamma_1^+(0)\gamma_2^+(0)\cdots\gamma_{n-1}^+(0)$.
However, in this case we deduce from (\ref{fact1}) that 
\[2^{n-1}d=\deg(\gamma_n^+(0))\leq\deg(\gamma_1^+(0)\gamma_2^+(0)\cdots\gamma_{n-1}^+(0))=d+ 2d+\dots 2^{n-2}d=(2^{n-1}-1)d.\]
But this inequality forces $d=0$, a contradiction. Therefore, for all $n\geq2$ a prime $\mathfrak{p}_n$ of $K=\mathbb{Q}(t)$ as in Theorem \ref{maxtest} exists. In particular, the argument in the proof Theorem \ref{maxtest} implies that the extensions $K_{\gamma,n}/K_{\gamma,n-1}$ are maximal for all $n\geq2$. Likewise, since $-\gamma_1^+(0)$ is not a square in $K$ (it's square-free), the extension $K_{\gamma,1}/K$ is also maximal. Hence, $G_{\gamma,K}=\Aut(T_\gamma)$ for all $\gamma\in\Phi_S$ as claimed.             
\end{proof} 
 
\end{document}